\documentclass[10pt, leqno]{article}

\usepackage{mathrsfs}
\usepackage{amsmath}
\usepackage{amssymb}
\usepackage{amsthm}
\usepackage{comment}

\usepackage{accents}

\newtheorem{theorem}{Theorem}[section]
\newtheorem{prop}{Proposition}[section]
\newtheorem{lem}{Lemma}[section]
\newtheorem{cor}{Corollary}[section]
\newtheorem{remark}{Remark}[section]

\numberwithin{equation}{section}

\def\ep{\epsilon}

\begin{document}

\title{Diagonal  symmetrizers for
hyperbolic operators with triple characteristics}

\author{Tatsuo Nishitani\footnote{Department of Mathematics, Osaka University:  
nishitani@math.sci.osaka-u.ac.jp
}}

\date{}
\maketitle

\def\dif{\partial}
\def\al{\alpha}
\def\be{\beta}
\def\ga{\gamma}
\def\om{\omega}
\def\lam{\lambda}
\def\tika{{\tilde \kappa}}
\def\baka{{\bar \kappa}}
\def\varep{\varepsilon}
\def\tal{{\tilde\alpha}}
\def\tbe{{\tilde\beta}}
\def\tis{{\tilde s}}
\def\bas{{\bar s}}
\def\R{{\mathbb R}}
\def\N{{\mathbb N}}
\def\C{{\mathbb C}}
\def\Q{{\mathbb Q}}
\def\Ga{\Gamma}
\def\La{\Lambda}
\def\lr#1{\langle{#1}\rangle}
\def\mD{\lr{ D}_{\mu}}
\def\xim{\lr{\xi}_{\mu}}

\begin{abstract}
Symmetrizers for hyperbolic operators are obtained by diagonalizing the B\'ezoutian matrix  of the principal symbols. Such diagonal symmetrizers are applied to the Cauchy problem for hyperbolic operators with triple characteristics. In particular, the Ivrii's conjecture concerned with effectively hyperbolic critical points is proved for differential operators with time dependent coefficients, also for third order differential operators with two independent variables with analytic coefficients.
\end{abstract}

\smallskip
 {\footnotesize Keywords: Symmetrizer, B\'ezoutiant, triple characteristic, Tricomi type, Cauchy problem.}
 
 \smallskip
 {\footnotesize Mathematics Subject Classification 2010: Primary 35L30, Secondary 35G10}

%%%%%%%%
\section{Introduction}

This paper is devoted to the Cauchy problem 
\begin{equation}
\label{eq:CPm}
\left\{\begin{array}{ll}
D_t^mu+\sum_{j=0}^{m-1}\sum_{|\alpha|+j\leq m}a_{j,\alpha}(t,x)D_x^{\alpha}D_t^ju=0,\\[8pt]
D_t^ju(0,x)=u_j(x),\quad j=0,\ldots,m-1
\end{array}\right.
\end{equation}
where $t\geq 0$, $x\in\R^n$ and the coefficients $a_{j,\alpha}(t,x)$ are real valued $C^{\infty}$ functions in a neighborhood of the origin of $\R^{1+n}$ and $D_x=(D_{x_1},\ldots,D_{x_n})$, $D_{x_j}=(1/i)(\dif/\dif x_j)$ and $D_t=(1/i)(\dif/\dif t)$. The problem is $C^{\infty}$ well-posed near the origin for $t\geq 0$ if one can find a $\delta>0$ and a neighborhood $U$ of the origin of $\R^n$ such that \eqref{eq:CPm} has a unique solution $u\in C^{\infty}([0,\delta)\times U)$ for any $u_j(x)\in C^{\infty}(\R^n)$. We assume that the principal symbol $p$ is hyperbolic for $t\geq 0$, that is
\[
p(t,x,\tau,\xi)=\tau^m+\sum_{j=0}^{m-1}\sum_{|\alpha|+j=m}a_{j,\alpha}(t,x)\xi^{\alpha}\tau^j
\]
has only real roots in $\tau$ for  $(t,x)\in [0,\delta')\times U'$ and $\xi\in \R^n$ with some $\delta'>0$ and a neighborhood $U'$ of the origin which is necessary in order that the Cauchy problem \eqref{eq:CPm} is $C^{\infty}$ well-posed near the origin for $t\geq 0$ (\cite{Lax}, \cite{Mbook}). 
 
In this paper we are mainly concerned with the case that the multiplicity of the characteristic roots is at most $3$. This implies that it is essential to study  
operators $P$ of the form
\begin{equation}
\label{eq:takata}
P=D_t^3+\sum_{j=1}^3a_j(t,x,D)\lr{D}^jD_t^{3-j}
\end{equation}
which is differential operator in $t$ with coefficients $a_{j}\in S^0$, classical pseudodifferential operator of order $0$, where $\lr{D} = {\rm Op}((1 + |\xi|^2)^{1/2})$. One can assume that $a_1(t,x,D)=0$ without loss of generality  and hence the principal symbol  has the form
\[
p(t, x, \tau, \xi) = \tau^3 -a(t, x, \xi)|\xi|^2\tau
- b(t, x, \xi)|\xi|^3. 
\]
With $U={^t}(D_t^2u,\lr{D}D_tu,\lr{D}^2u)$ the equation $P u=f$ is reduced to 
\begin{equation}
\label{eq:redE}
D_tU=A(t,x,D)\lr{D}U+B(t,x,D)U+F
\end{equation}
where $A, B\in S^0$, $F={^t}(f,0,0)$ and
\[
 A(t,x,\xi)=
\begin{bmatrix}0&a&b\\
1&0 &0\\
0&1&0
\end{bmatrix}.
\]
Let $S$ be  the B\'ezoutiant of $p$ and $\dif p/\dif \tau$, that is
\begin{equation}
\label{eq:defS}
S(t,x,\xi)=
\begin{bmatrix}3&0&-a\\
0&2a&3b\\
-a&3b&a^2
\end{bmatrix}
\end{equation}
then $S$ is nonnegative definite and symmetrizes $A$, that is $SA$ is symmetric which is easily examined directly, though this is a special case of a general fact (see \cite{Ja2}, \cite{N4}).  Then one of the most important works would be to obtain lower bound of $({\rm Op}(S)U,U)$. The sharp G\aa rding inequality (\cite{LaNi}, \cite{Hobook}) gives a lower bound
\[
{\mathsf{Re\,}}({\rm Op}(S)U,U)\geq -C\|\lr{D}^{-1/2}U\|^2
\]
which is, in general,  too weak to study the Cauchy problem for general weakly hyperbolic operator $P$, in particular the well posed Cauchy problem with loss of derivatives, although  
applying this symmetrizer many interesting results are obtained by several authors, see for example \cite{JaTa}, \cite{DaSpa}, \cite{KiSpa}, \cite{GaR:1}, \cite{GaR:2},  \cite{NP}. In these works  one of the main points is  how one can derive a suitable lower bound of ${\rm Op}(S)$ from the hyperbolicity condition assumed on $p$, that is
\begin{equation}
\label{eq:hyp:cond}
\Delta=4\,a(t,x,\xi)^3-27\,b(t,x,\xi)^2\geq 0, \quad (t,x,\xi)\in [0,T)\times U\times \R^n.
\end{equation}

In this paper we employ a new idea which is to diagonalize $S$ by an orthogonal matrix $T$  so that $T^{-1}ST=\Lambda={\rm diag}\,(\lambda_1,\lambda_2,\lambda_3)$  where $0\leq \lambda_1\leq \lambda_2\leq \lambda_3$ are the eigenvalues of $S$ and reduce the equation to that of $V=T^{-1}U$; roughly
\begin{equation}
\label{eq:daitai}
D_tV=A^T\lr{D}V+B^TV,\quad A^T=T^{-1}AT
\end{equation}
where $\Lambda$  symmetrizes $A^T$. For general nonnegative definite symmetric $S$ 
it seems that we have nothing new but our $S$ is a special one which is the B\'ezoutiant  of hyperbolic polynomial $p$ and $\dif p/\dif \tau$. Indeed, as we will see in Section \ref{sec:DiaSym}, one has
\[
\frac{\Delta}{a}\preceq \lambda_1\preceq a^2,\quad \lambda_2\simeq a,\quad \lambda_3\simeq 1.
\]
Since \eqref{eq:daitai} is a symmetrizable system with a diagonal symmetrizer $\Lambda$, a natural energy will be
\[
{\mathsf {Re}}\,\big({\rm Op}(\Lambda)U,U\big)=\sum_{j=1}^3\big({\rm Op}(\lambda_j)U_j,U_j\big)
\]
and it could be expected that  scalar operators ${\rm Op}(\lambda_j)$ reflect the hyperbolicity condition \eqref{eq:hyp:cond} quite directly.

If $p=0$ has a triple characteristic root $\tau=0$ at $(0,x,\xi)$ such that $(0,x,\tau,\xi)$ is effectively hyperbolic (see Section \ref{sec:app:tdep})  then one sees that $\dif_ta(0,x,\xi)>0$ and hence  $\dif_t^3\Delta(0,x,\xi)>0$, which follows from \eqref{eq:hyp:cond}, so that essentially $a$ and $\Delta$ are polynomials in $t$ of degree $1$ and $3$ respectively. Then we see that $\Delta/a\preceq\lambda_1$ behaves like a second order polynomial in $t$ which is nonnegative for $t\geq 0$. Finding a finite number of functions $\phi_j$ such that $\dif_t\phi_j>0$ and $\phi_j^2\preceq \Delta/a$  we estimate the weighted energy ${\mathsf {Re}}\,\big({\rm Op}(\phi_j^N\Lambda)U,U\big)$ with a suitable $N\in\R$. In Section \ref{sec:app:tdep}  this procedure is carried out for operators of order $m$ with effectively hyperbolic critical points with time dependent coefficients, and it is proved that the Cauchy problem is $C^{\infty}$ well-posed for any lower order term. In Section \ref{sec:app:xdep} the same assertion is proved for third order operators with two independent variables with analytic coefficients, so that Ivrii's conjecture is proved for these operators.

In Section \ref{sec:enehyoka:1}, admitting the existence of such a weight function, we explain how to derive energy estimates.  To do so we need to estimate the derivatives of $\Lambda$ and $A^T$, essentially those of $\lambda_j$, which  is done in Section \ref{sec:DiaSym}. 
 
In the last section we show that the same idea is applicable to  hyperbolic operators  with more general triple characteristics, utilizing a  homogeneous third order operator with two independent variables.

%%%%%%%%%%%%%%%%%
\section{Daiagonal symmetrizers }
\label{sec:DiaSym}

Consider 
\[
p(\tau,t,X)=\tau^3-a(t,X)\tau-b(t,X)
\]
where $a(t,X)$ and $b(t,X)$ are real valued and $C^{\infty}$ in $(t,X)\in (-c,T)\times W$ with bounded derivatives of all order where $W$ is an open set in $ \R^l$ such that ${\bar X}\in W$ and $c>0$ is some positive constant. Assume 
\begin{equation}
\label{eq:cond:1}
\Delta(t,X)=4\,a(t,X)^3-27\,b(t,X)^2\geq 0,\; (t,X)\in [0,T)\times W, \;\;a(0,\bar{X})=0
\end{equation}
that is, $p(\tau,t,X)=0$ has only real roots for $(t,X)\in [0,T)\times W$ and  has a triple root $\tau=0$ at $(0,{\bar X})$. Moreover  assume that there is  no triple root in $t>0$;
\begin{equation}
\label{eq:eq:cond:1bis}
a(t,X)>0, \quad (t,X)\in (0,T)\times W.
\end{equation}
 Denote
\begin{equation}
\label{eq:esu}
S(t,X)=
\begin{bmatrix}3&0&-a\\
0&2\,a&3\,b\\
-a&3\,b&\,a^2
\end{bmatrix},\qquad A(t,X)=
\begin{bmatrix}0&a&b\\
1&0&0\\
0&1&0
\end{bmatrix}
\end{equation}
then $S$ is nonnegative definite and $S(t,X)A(t,X)$ is symmetric. Let
\[
0\leq \lambda_1(t,X)\leq \lambda_2(t,X)\leq \lambda_3(t,X)
\]
be the eigenvalues of $S(t,X)$. 

%%%%%%%%%%%%%%%%%%%%%
\subsection{Behavior of eigenvalues}

We show
\begin{prop}
\label{pro:Skon}There exist a neighborhood ${\mathcal U}$ of $(0,\bar{X})$ and $K>0$ such that
\begin{align}
\begin{split}\label{eq:kon:1}
\Delta/(6a+2a^2+2a^3)\leq \lambda_1\leq \big(2/3+Ka\big)\,a^2,
\end{split}\\
\begin{split}\label{eq:kon:2}
(2-Ka)\,a\leq \lambda_2\leq (2+Ka)\,a,
\end{split}\\
\begin{split}\label{eq:kon:3}
3\leq \lambda_3\leq 3+Ka^2
\end{split}
\end{align}
for $(t,X)\in {\mathcal U}\cap \{t>0\}$.
\end{prop}
\begin{cor}
\label{cor:konname}There exists a neighborhood ${\mathcal U}$ of $(0,\bar{X})$ such that
\[
\lambda_i(t,X)\in C^{\infty}({\mathcal U}\cap \{t>0\}),\quad i=1,2,3.
\]
\end{cor}
\begin{proof}Recalling $a(0,\bar{X})=0$ from Proposition \ref{pro:Skon} one can choose ${\mathcal U}$ such that
\[
 \lambda_1<\lambda_2<\lambda_3\quad \text{in}\;\;{\mathcal U}\cap\{t>0\}
\]
then the assertion follows immediately from the Implicit function theorem.
\end{proof}
\begin{remark}
\label{rem:jyukon}\rm It may happen $\Delta(t,X)=0$ for $t>0$ so that $p(\tau,t,X)=0$  has a double root $\tau$ at $(t,X)$ while $\lambda_i(t,X)$ are smooth there.
\end{remark}

\noindent
Proof of Proposition \ref{pro:Skon}: Denote $q(\lambda)={\rm det}\,(\lambda I-S)$;
\begin{equation}
\label{eq:qkata}
q(\lambda)=\lambda^3-(3+2a+a^2)\lambda^2+(6a+2a^2+2a^3-9b^2)\lambda-\Delta.
\end{equation}
Let $\mu_1\leq\mu_2$ be the roots of $q_{\lambda}=\partial q/\partial \lambda=0$ and hence
\[
\lambda_1\leq \mu_1\leq \lambda_2\leq \mu_2\leq \lambda_3.
\]
It is easy to see $\mu_1=a(1+O(a))$ and $ \mu_2=2+O(a)$ which gives
\begin{equation}
\label{eq:kon3}
\lambda_1\leq a(1+O(a)),\quad \lambda_3\geq 2+O(a).
\end{equation}
In the $(\lambda,\eta)$ plane, the tangent line of the curve $\eta=q(\lambda)$ at $(0,q(0))$ intersects with $\lambda$ axis at $(\Delta/q_{\lambda}(0),0)$ and hence
\[
\lambda_1(t,X)\geq \Delta/q_{\lambda}(0).
\]
Since $q_{\lambda}(0)\leq 6a+2a^2+2a^3$ the left inequality of \eqref{eq:kon:1} is obvious. Compute $q(\delta a^2)$ with $\delta>0$. Since $2a^3-9b^2=\Delta/2+9b^2/2\geq 0$  one has
\begin{align*}
q(\delta a^2)
\geq \delta^3a^6-\delta^2a^4(3+2a+a^2)+\delta a^2(6a+2a^2)-4a^3+27b^2\\
\geq a^3\Big\{(6\delta-4)+\delta(2-3\delta)a-2\delta^2a^2+\delta^2(\delta-1)a^3\Big\}.
\end{align*}
Here we take $\delta=2/3+Ka$ then noting $a(0,\bar{X})=0$ one can choose a neighborhood ${\mathcal U}$ of $(0,\bar{X})$ such that
\[
q(\delta a^2)\geq a^4\big\{K-3K\delta a-2\delta^2 a+\delta^2(\delta-1)a^2\big\}>0
\]
for $(t,X)\in {\mathcal U}\cap\{t>0\}$. This proves that $\lambda_1\leq \delta a^2$ and hence the right inequality of \eqref{eq:kon:1}. Turn to $\lambda_2$. Consider $q(\delta a)$ with $\delta>0$ again. Note
\begin{align*}
q(\delta a)\geq a^2\Big\{\delta^3a-\delta^2(3+2a+a^2)+\delta(6+2a)-4a\Big\}+27b^2\\
\geq a^2\Big\{3\delta(2-\delta)+(\delta^3-2\delta^2+2\delta-4)
a-\delta^2
a^2\Big\}
\end{align*}
and choose $\delta=2-Ka$ which gives
\[
q(\delta a)\geq a^3\Big\{6K-(3K^2+2K+K\delta^2-\delta^2)a\Big\}.
\]
Therefore for any $K>0$ one can find ${\mathcal U}$ such that $q(\delta a)>0$ in ${\mathcal U}\cap\{t>0\}$. Since one can assume $\delta a<\lambda_3$ by \eqref{eq:kon3} then $\delta a\in (\lambda_1,\lambda_2)$ which proves the left inequality of \eqref{eq:kon:2}. Repeating similar arguments one gets
\begin{align*}
q(\delta a)\leq a^2\Big\{3\delta(2-\delta)
+(4+2\delta+\delta^3-2\delta^2)
a+\delta(2-\delta)a^2\Big\}
\end{align*}
because $27b^2\leq 4 a^3$ in $t\geq 0$. Taking $\delta=2+Ka$ one has
\[
q(\delta a)\leq a^3\Big\{(
8-6K)+(2K+\delta^2 K-\delta Ka-3K^2)
a\Big\}.
\]
Fixing any $K>4/3$ one can find ${\mathcal U}$ such that $q(\delta a)<0$ in ${\mathcal U}\cap\{t>0\}$. Since $\lambda_1<\delta a$ thanks to \eqref{eq:kon:1} one concludes $(2+Ka)a\in (\lambda_2,\lambda_3)$ which shows the right inequality of \eqref{eq:kon:2}. Finally we check \eqref{eq:kon:3}. It is easy to see that
\[
q(3)=a^2(-3+2a)<0
\]
in ${\mathcal U}\cap\{t>0\}$ if ${\mathcal U}$ is small 
so that $3\leq \lambda_3$ in ${\mathcal U}\cap\{t>0\}$. Note that
\begin{align*}
q(\delta)\geq \delta\Big\{\delta(\delta-3)+(6-2\delta)a+(2-\delta^2)a^2\Big\}-4a^3
\end{align*}
where we take $\delta=3+Ka^2$ so that
\begin{align*}
q(3+Ka^2)=a^2\Big\{3(3K-1)-(6K+4)a+3(K^2-k)a^2\Big\}\\
+Ka^4\Big\{(3K-1)-2Ka+(K^2-K)a^2\Big\}.
\end{align*}
Thus fixing any $K>1/3$ one can find ${\mathcal U}$ such that $q(3+Ka^2)>0$ in ${\mathcal U}\cap\{t>0\}$. Since $3+Ka^2>\lambda_2$ which proves the right inequality of \eqref{eq:kon:3}.
\qed
%
%%%%%%%%%%%%%%%%%%%%%%
\subsection{Behavior of eigenvectors}
\label{sec:defT}

If we write $n_{ij}$ for the $(i,j)$-cofactor of $\lambda_kI-S$ then $^t(n_{j1},n_{j2},n_{j3})$ is, if non-trivial, an eigenvector corresponding to $\lambda_k$. We take $k=1$, $j=3$ and hence
\[
\begin{bmatrix}a(2\,a-\lambda_1)\\
3\,b(\lambda_1-3)\\
(\lambda_1-3)(\lambda_1-2\,a)
\end{bmatrix}=\begin{bmatrix}\ell_{11}\\
\ell_{21}\\
\ell_{31}
\end{bmatrix}
\]
is an eigenvector corresponding to $\lambda_1$ and therefore
\[
{\bf t}_1=\begin{bmatrix}t_{11}\\
t_{21}\\
t_{31}
\end{bmatrix}=\frac{1}{d_1}\begin{bmatrix}\ell_{11}\\
\ell_{21}\\
\ell_{31}
\end{bmatrix},\quad d_1=\sqrt{\ell_{11}^2+\ell_{21}^2+\ell_{31}^2}
\]
is a normalized eigenvector corresponding to $\lambda_1$. Thanks to Proposition \ref{pro:Skon} and $b=O(a^{3/2})$ it is clear that there is $C>0$ such that
\begin{equation}
\label{eq:dno:1}
a/C\leq d_1\leq C\,a,\quad \text{in}\quad {\mathcal U}\cap\{t>0\}.
\end{equation}
Similarly choosing $k=2, j=2$ and $k=3, j=1$ 
\[
\begin{bmatrix}-3\,ab\\
(\lambda_2-3)(\lambda_2-a^2)-a^2\\
3\,b(\lambda_2-3)
\end{bmatrix}=\begin{bmatrix}\ell_{12}\\
\ell_{22}\\
\ell_{32}
\end{bmatrix},\quad \begin{bmatrix}(\lambda_3-2a)(\lambda_3-a^2)-9b^2\\
-3ab\\
-a(\lambda_3-2a)
\end{bmatrix}=\begin{bmatrix}\ell_{13}\\
\ell_{23}\\
\ell_{33}
\end{bmatrix}
\]
are eigenvectors corresponding to $\lambda_2$ and $\lambda_3$ respectively and
\[
{\bf t}_j=\begin{bmatrix}t_{1j}\\
t_{2j}\\
t_{3j}
\end{bmatrix}=\frac{1}{d_j}\begin{bmatrix}\ell_{1j}\\
\ell_{2j}\\
\ell_{3j}
\end{bmatrix},\quad d_j=\sqrt{\ell_{1j}^2+\ell_{2j}^2+\ell_{3j}^2}
\]
are normalized eigenvectors corresponding to $\lambda_j$, $j=2,3$. Thanks to Proposition \ref{pro:Skon} there is $C>0$ such that
\begin{equation}
\label{eq:dno:23}
a/C\leq d_2\leq C\,a,\quad 1/C\leq d_3\leq C.
\end{equation}
Denote $T=({\bf t}_1,{\bf t}_2, {\bf t}_3)=(t_{ij})$ then $T$ is an orthogonal matrix, ${^t}TT=I$, smooth in $(t,X)\in {\mathcal U}\cap\{t>0\}$ which diagonalizes $S$;
\[
\Lambda=T^{-1}ST={^t}TST=\begin{bmatrix}\lambda_1&0&0\\
0&\lambda_2&0\\
0&0&\lambda_3
\end{bmatrix}.
\]
Note that $\Lambda$ symmetrizes $A^T=T^{-1}AT$;
\[
^{t}(\Lambda A^T)=^{t}\!\!(\,^{t}TSAT)=^{t}\!\!T\,^{t}\!(SA)T={^t}TSAT=\Lambda A^T.
\]
Denote $A^T=({\tilde a}_{ij})$. Since $\Lambda A^T$ is symmetric ${\tilde a}_{ij}$ satisfies
\[
{\tilde a}_{21}=\frac{\lambda_1{\tilde a}_{12}}{\lambda_2},\quad {\tilde a}_{31}=\frac{\lambda_1{\tilde a}_{13}}{\lambda_3},\quad {\tilde a}_{32}=\frac{\lambda_2{\tilde a}_{23}}{\lambda_3}
\]
which shows that ${\tilde a}_{21}=O(a^{-1}\lambda_1){\tilde a}_{12}$, ${\tilde a}_{31}=O(\lambda_1){\tilde a}_{13}$ and ${\tilde a}_{32}=O(a){\tilde a}_{23}$.
%Therefore writing $A^T=({\tilde a}_{ij})$ it is clear that one can write
%
%\begin{equation}
%\label{eq:DA}
%DA^T=\begin{bmatrix}\lambda_1{\tilde a}_{11}&\lambda_1{\tilde a}_{12}&\lambda_1{\tilde a}_{13}\\
%\lambda_1{\tilde a}_{12}&\lambda_2{\tilde a}_{22}&\lambda_2{\tilde a}_{23}\\
%\lambda_1{\tilde a}_{13}&\lambda_2{\tilde a}_{23}&\lambda_3{\tilde a}_{33}\\
%\end{bmatrix}.
%\end{equation}
%
%%
%\begin{remark}
%\label{rem:sankaku}\rm
%Since $\Lambda A^T$ is symmetric it follows that if $\lambda_1=0$ and $0<\lambda_2\leq\lambda_3$ then ${\tilde a}_{21}={\tilde a}_{31}=0$ so that $A^T$ is block triangular.
%\end{remark}
%
Finally in view of \eqref{eq:dno:1}, \eqref{eq:dno:23} and Proposition \ref{pro:Skon} it is easy to check that
\begin{equation}
\label{eq:ordT}
T=\big({\bf t}_1,{\bf t}_2,{\bf t}_3\big)=\begin{bmatrix}O(a)&O(a^{3/2})&O(1)\\
O(\sqrt{a})&O(1)&O(a^{5/2})\\
O(1)&O(\sqrt{a})&O(a)
\end{bmatrix}
\end{equation}
near $(t,X)=(0,{\bar X})$.

%%%%%%%%%%%%%%%%%%%%
\subsection{Smoothness  of eigenvalues}

First recall \cite[Lemma 3.2]{NP}
\begin{lem}
\label{lem:NiP}Assume \eqref{eq:cond:1}. Then
\begin{align*}
|\partial_X^{\alpha}a|\preceq \sqrt{a},\quad |\partial_X^{\alpha}b|\preceq a,\quad |\partial_tb|\preceq \sqrt{a}
\end{align*}
for $|\alpha|=1$ and $(t,X)\in (0,T)\times W$.
\end{lem}
We show
\begin{lem}
\label{lem:bibun} For $|\alpha|=1$ one has
\begin{equation}
\label{eq:bibun:a}
|\partial_X^{\alpha}\lambda_1|\preceq a^{3/2},\quad |\partial_X^{\alpha}\lambda_2|\preceq \sqrt{a},\quad |\partial_X^{\alpha}\lambda_1|\preceq \sqrt{a}.
\end{equation}
\end{lem}
\begin{proof} 
Since
\[
\dif_X^{\alpha}q(\lambda)=-\dif_X^{\alpha}(2a+a^2)\lambda^2+\dif_X^{\alpha}(6a+2a^2+2a^3-9b^2)\lambda-\dif_X^{\alpha}(4a^3-27b^2)
\]
it follows from Lemma \ref{lem:NiP} that
\begin{align*}
\big|\dif_X^{\alpha}q(\lambda_j)\big|\preceq |\dif_X^{\alpha}a|\lambda_j^2+\big(|\dif_X^{\alpha}a|+|b||\dif_X^{\alpha}b|\big)\lambda_j+\big(a^2|\dif_X^{\alpha}a|+|b||\dif_X^{\alpha}b|\big)\\
\preceq |\dif_X^{\alpha}a|\lambda_j+a^{5/2}\preceq \sqrt{a}\,\lambda_j+a^{5/2}.
\end{align*}
From $
q_{\lambda}(\lambda_j)\partial_X^{\alpha}\lambda_j+\partial_X^{\alpha}q(\lambda_j)=0$ 
one has
\[
|\dif_X^{\alpha}\lambda_j|\preceq \frac{\sqrt{a}\,\lambda_j+a^{5/2}}{|q_{\lambda}(\lambda_j)|}.
\]
Noting $q_{\lambda}(\lambda_j)=\prod_{k\neq j}(\lambda_j-\lambda_k)$ one sees 
\begin{equation}
\label{eq:dmod}
q_{\lambda}(\lambda_j)\simeq a \;\;\text{for}\;\;j=1,2,\quad  q_{\lambda}(\lambda_3)\simeq 1
\end{equation}
 thanks to Proposition \ref{pro:Skon} and hence the assertion.
\end{proof}
Next estimate $\dif_t\lambda_j$. 
\begin{lem}
\label{lem:bibun:t} Assuming \eqref{eq:cond:1} one has
\begin{equation}
\label{eq:bibun:b}
|\dif_t\lambda_1|\preceq a,\quad |\dif_t\lambda_2|\preceq 1,\quad |\dif_t\lambda_3|\preceq 1.
\end{equation}
\end{lem}
\begin{proof}Repeating the same arguments in the proof of Lemma \ref{lem:bibun} one has
\[
|\dif_tq(\lambda_j)|\preceq |\dif_ta|\lambda_j+a^2|\dif_ta|+|b||\dif_tb|\preceq \lambda_j+a^2
\]
which proves the assertion.
\end{proof}

%
%%%%%%%%%%%%%%%%
\section{How to apply diagonal symmetrizers}
\label{sec:enehyoka:1}

 Taking   
\[
Pu=\dif_t^3u-a(t,x)\dif_x^2\dif_tu-b(t,x)\dif_x^3u
\]
with one space variable $x\in\R$, we explain how to apply diagonal symmetrizers constructed in preceding sections. 
Assume that
\begin{equation}
\label{eq:cond:b}
\Delta(t,x)=4\,a(t,x)^3-27\,b(t,x)^2\geq 0\quad (t,x)\in [0,T)\times W
\end{equation}
and $a(0,0)=0$ such that $p(\tau,0,0,1)=0$ has the triple root $\tau=0$ where $W$ is an open interval containing the origin. In what follows we work in a region where $a(t,x)>0$. With $U=(\dif_t^2u,\dif_x\dif_tu,\dif_x^2u)$ the equation $Pu=f$ is reduced to
\begin{equation}
\label{eq:kei:rei}
\dif_tU=A(t,x)\dif_xU+F,\quad A=\begin{bmatrix}0&a&b\\
                                                1&0&0\\
                                                0&1&0\\
                                                \end{bmatrix},\quad F=\begin{bmatrix}f\\
                                                0\\
                                                0\\
                                                \end{bmatrix}.
\end{equation}
Then $S$ given by \eqref{eq:esu} symmetrizes $A$, and $T$ given by \eqref{eq:ordT} diagonalizes $S$. So we set $V=T^{-1}U$ and  rewrite  the equation \eqref{eq:kei:rei} to
\begin{equation}
\label{eq:kei:rei:2}
\dif_tV=A^T\dif_xV+\big((\dif_tT^{-1})T-A^T(\dif_xT^{-1})T\big)V+T^{-1}F
\end{equation}
where $A^T=T^{-1}AT$. To simplify notation let us write \eqref{eq:kei:rei:2} with $f=0$ as
\[
\dif_tV={\mathcal A}\dif_xV+{\mathcal B}V
\]
with ${\mathcal A}=A^T$ and ${\mathcal B}=(\dif_tT^{-1})T-{\mathcal A}(\dif_xT^{-1})T$. 
% 

%%%%%%%%%%%%%%%%%
\subsection{Energy with scalar weight}
\label{sec:howapp}

Consider an energy with a scalar weight  $\phi(t,x)>0$ with $\dif_t\phi=1$ and $|\dif_x\phi|\preceq 1$;
\[
(\phi^{-N} \Lambda V,V)=\int \phi^{-N}\lr{\Lambda V,V}dx=\sum_{j=1}^3\int \phi^{-N}\lambda_j |V_j|^2\,dx
\]
where $\lr{V,W}$ stands for the inner product in $\C^3$ and $N>0$ is a positive parameter. In what follows we assume that $V(t,x)$ has small support in $x$. Note that 
\begin{align*}
\frac{d}{dt}\,(\phi^{-N} \Lambda V,V)=-N\big(\phi^{-N-1}\Lambda V,V\big)
+\big(\phi^{-N}(\dif_t\Lambda)V,V\big)\\
+2\,{\mathsf{Re}}\,\big(\phi^{-N}\Lambda({\mathcal A}\dif_xV+{\mathcal B}V),V\big).
\end{align*}
%
%where
%
%\begin{equation}
%\label{eq:ene}
%N\phi^{-N-1}\lr{DV,V}=N\sum_{j=1}^3\phi^{-N-1}\lambda_j |V_j|^2.
%\end{equation}
%
Since $\Lambda{\mathcal A}$ is symmetric and hence
\begin{equation}
\label{eq:DAdifV}
2\,{\mathsf{Re}}\,(\phi^{-N}\Lambda{\mathcal A}\dif_x V,V)=N\big(\phi^{-N-1}(\dif_x\phi)\Lambda{\mathcal A}V,V\big)-\big(\phi^{-N}\dif_x(\Lambda{\mathcal A})V,V).
\end{equation}
As for a scalar weight $\phi$ we assume
\begin{equation}
\label{eq:katei:D}
\phi^2\,a\preceq \Delta,\qquad \phi\big|\dif_t\Delta\big|\preceq \Delta,\qquad \phi\big|\dif_ta\big|\preceq a.
\end{equation}
\begin{lem}
\label{lem:katei:D}The assumption \eqref{eq:katei:D} implies
\begin{equation}
\label{eq:katei:2}
\phi^2\preceq \lambda_1,\qquad \phi\big|\dif_t\lambda_1\big|\preceq \lambda_1,\qquad \phi\big|\dif_t\lambda_2\big|\preceq \lambda_2.
\end{equation}
\end{lem}
\begin{proof} In view of Proposition \ref{pro:Skon} the assertion $\phi^2\preceq \lambda_1$ is clear. Note that from $|\dif_tq(\lambda_i)|\preceq (|\dif_ta|+|b||\dif_tb|)\lambda_i+|\dif_t\Delta|$ and Lemma \ref{lem:NiP} it follows that
\[
\big|\dif_t\lambda_i\big|\preceq \frac{\big(|\dif_ta|+a^2\big)\lambda_i+|\dif_t\Delta|}{|q_{\lambda}(\lambda_i)|}.
\]
Taking \eqref{eq:dmod} into account one has
\[
|\dif_t\lambda_i|\preceq \frac{|\dif_ta|}{a}\lambda_i+a\,\lambda_i+\frac{|\dif_t\Delta|}{a},\quad i=1,2
\]
which implies $\phi|\dif_t\lambda_1|\preceq \lambda_1$ thanks to \eqref{eq:katei:D} and \eqref{eq:kon:1}. As for $\lambda_2$ noting that $|\dif_t\Delta|\preceq a^2$ by Lemma \ref{lem:NiP} the assertion follows immediately from \eqref{eq:katei:D} and \eqref{eq:kon:2}.
\end{proof}
%

%%%%%%%%%%%%%%%%%%%
\subsection{Estimate of energy, terms $\lr{(\dif_t\Lambda)V,V}$, $\lr{(\dif_x\phi)\Lambda{\mathcal A}V,V}$}

Thanks to \eqref{eq:katei:2}, $\big|\phi^{-N}\lr{(\dif_t\Lambda)V,V}\big|$ is bounded by $N\phi^{-N-1}\lr{\Lambda V,V}$ taking $N$ large. On the other hand,  from Lemma \ref{lem:difTAT} below  it follows that
\[
\Lambda{\mathcal A}=\begin{bmatrix}
O(\lambda_1\sqrt{a})&O(\lambda_1)&O(\lambda_1\sqrt{a})\\
O(a^2)&O(a^{3/2})&O(a)\\
O(a^{3/2})&O(a)&O(a^{5/2})\\
\end{bmatrix}.
\]
Recalling that $\Lambda{\mathcal A}$ is symmetric it is clear that $\big|\lr{\Lambda{\mathcal A}V,V}\big|$ is bounded by
\begin{align*}
\sqrt{a}(\lambda_1|V_1|^2+a|V_2|^2+a^2|V_3|^2)+\lambda_1\,|V_1|\,|V_2|
+\lambda_1\sqrt{a}\,|V_1|\,|V_3|
+a|V_2||V_3|.
\end{align*}
Since $\lambda_1\preceq \sqrt{\lambda_1}\,a$ and hence $\lambda_1|V_1|\,|V_2|\preceq \sqrt{a}\,(\lambda_1|V_1|^2+a\,|V_2|^2)$  it follows that
\begin{equation}
\label{eq:adphi}
 \big|N\phi^{-N-1}(\dif_x\phi)\lr{\Lambda{\mathcal A}V,V}\big|\leq CN\sqrt{a}\,|\dif_x\phi|\,\sum_{j=1}^3\phi^{-N-1}\lambda_j|V_j|^2
 \end{equation}
 with some $C>0$. In a small neighborhood of $(0,0)$ where $a$ is enough small one can bound the right-hand side  by $N\phi^{-N-1}\lr{\Lambda V,V}$.
 
%%%%%%%%%%%%%%%%%%%%
\subsection{Estimate of energy, term $\lr{\Lambda{\mathcal B}V,V}$}

Recall
\begin{align*}
\lr{\Lambda{\mathcal B}V,V}
=\lr{\Lambda(\dif_tT^{-1})TV,V}-\lr{\Lambda{\mathcal A}(\dif_xT^{-1})TV, V}.
\end{align*}
Applying Lemmas \ref{lem:bibun} and \ref{lem:bibun:t} we estimate $(\dif_tT^{-1})T$ and $(\dif_xT^{-1})T$. First note that
\[
(\dif_tT^{-1})T=(\dif_t(^{t}T))T=( \lr{\dif_t {\bf t}_i,{\bf t}_j})
\]
and $\lr{\dif_t {\bf t}_i,{\bf t}_j}=-\lr{{\bf t}_i, \dif_t{\bf t}_j}=-\lr{ \dif_t {\bf t}_j,{\bf t}_i}$ so that $(\dif_tT^{-1})T$ is antisymmetric. Note that
\begin{equation}
\label{eq:d:titj}
\lr{\dif_t {\bf t}_i,{\bf t}_j}=\frac{1}{d_id_j}\sum_{k=1}^3\dif_t\ell_{ki}\,\cdot{\bar \ell}_{kj}=\frac{1}{d_i}\sum_{k=1}^3\dif_t\ell_{ki}\,\cdot{\bar t}_{kj}
\end{equation}
because $\sum_{k=1}^3\ell_{ki}\,{\bar \ell}_{kj}=0$ if $i\neq j$. 
Thanks to Proposition \ref{pro:Skon} and Lemmas \ref{lem:bibun:t}, \ref{lem:NiP}  it follows that
\begin{equation}
\label{eq:hyoka:el:zero}
\begin{split}
|\ell_{11}|\preceq a^2,\quad |\ell_{21}|\preceq a^{3/2},\quad |\ell_{31}|\preceq a,\\
|\ell_{12}|\preceq a^{5/2},\quad |\ell_{22}|\preceq a,\quad |\ell_{32}|\preceq a^{3/2},\\
|\ell_{13}|\preceq 1,\quad |\ell_{23}|\preceq a^{5/2},\quad |\ell_{33}|\preceq a
\end{split}
\end{equation}
and that
\begin{equation}
\label{eq:hyoka:el:t}
\begin{split}
|\dif_t\ell_{11}|\preceq a,\quad |\dif_t\ell_{21}|\preceq \sqrt{a},\quad |\dif_t\ell_{31}|\preceq 1,\\
|\dif_t\ell_{12}|\preceq a^{3/2},\quad |\dif_t\ell_{22}|\preceq 1,\quad |\dif_t\ell_{32}|\preceq \sqrt{a},\\
|\dif_t\ell_{13}|\preceq 1,\quad |\dif_t\ell_{23}|\preceq a^{3/2},\quad |\dif_t\ell_{33}|\preceq 1.
\end{split}
\end{equation}
Therefore taking \eqref{eq:ordT}, \eqref{eq:d:titj} and \eqref{eq:hyoka:el:t} into account one obtains
\begin{equation}
\label{eq:dT:-1:T}
(\dif_tT^{-1})T=\begin{bmatrix}0&O(1/\sqrt{a})&O(1)\\
O(1/\sqrt{a})&0&O(\sqrt{a})\\
O(1)&O(\sqrt{a})&0\\
\end{bmatrix}.
\end{equation}
In order to estimate $\big|\phi^{-N}\lr{\Lambda(\dif_tT^{-1})TV,V}\big|$,  noting $\Lambda\simeq {\rm diag}(\lambda_1,a,1)$ and $\lambda_1\preceq a^2$, it suffices to estimate
\[
\sqrt{a}\,|V_1|| V_2|+|V_1||V_3|+\sqrt{a}\,|V_2||V_3|.
\]
Note that
\begin{align*}
\sqrt{a}\,|V_1||V_2|\preceq 
\phi^{-1}\lambda_1 |V_1|^2+a\frac{\phi}{\lambda_1} |V_2|^2
\preceq \phi^{-1}(\lambda_1 |V_1|^2+a |V_2|^2)
\end{align*}
because $\phi/\lambda_1\preceq  1/\phi$ by \eqref{eq:katei:2}. As for $|V_1||V_3|$ one has
\begin{align*}
|V_1||V_3|\preceq \phi^{-1}\lambda_1|V_1|^2+(\phi /\lambda_1)|V_3|^2
\preceq \phi^{-1}(\lambda_1|V_1|^2+|V_3|^2).
\end{align*}
Finally since $\sqrt{a}\,|V_2||V_3|\preceq a\,|V_2|^2+|V_3|^2$ 
one concludes that $\big|\phi^{-N}\lr{\Lambda(\dif_tT^{-1})TV,V}\big|$ is bounded by $N\phi^{-N-1}\lr{\Lambda V,V}$ taking $N$ large.

Turn to $\big|\phi^{-N}\lr{\Lambda{\mathcal A}(\dif_xT^{-1})TV, V}\big|$. 
From Proposition \ref{pro:Skon} and Lemmas \ref{lem:bibun}, \ref{lem:NiP} one has
\begin{equation}
\label{eq:hyoka:el:x}
\begin{split}
|\dif_x\ell_{11}|\preceq a^{3/2},\quad |\dif_x\ell_{21}|\preceq a,\quad |\dif_x\ell_{31}|\preceq \sqrt{a},\\
|\dif_x\ell_{12}|\preceq a^{2},\quad |\dif_x\ell_{22}|\preceq \sqrt{a},\quad |\dif_x\ell_{32}|\preceq a,\\
|\dif_x\ell_{13}|\preceq 1,\quad |\dif_x\ell_{23}|\preceq a^{2},\quad |\dif_x\ell_{33}|\preceq \sqrt{a}
\end{split}
\end{equation}
from which one concludes 
\begin{equation}
\label{eq:dx:TtT:a}
\dif_xT^{-1}={^t}(\dif_xT)=\begin{bmatrix}
O(\sqrt{a})&O(1)&O(1/\sqrt{a})\\
O(a)&O(1/\sqrt{a})&O(1)\\
O(1)&O(a^2)&O(\sqrt{a})\\
\end{bmatrix}
\end{equation}
and hence 
\begin{equation}
\label{eq:dx:TtT}
(\dif_xT^{-1})T=\begin{bmatrix}0&O(1)&O(\sqrt{a})\\
O(1)&0&O(a)\\
O(\sqrt{a})&O(a)&0\\
\end{bmatrix}.
\end{equation}
Here note that
\begin{lem}
\label{lem:difTAT} One has
\[
{\mathcal A}=\begin{bmatrix}
O(\sqrt{a})&O(1)&O(\sqrt{a})\\
O(a)&O(\sqrt{a})&O(1)\\
O(a^{3/2})&O(a)&O(a^{5/2})\\
\end{bmatrix},\;\; 
\dif_x{\mathcal A}=\begin{bmatrix}
O(1)&O(1/\sqrt{a})&O(1)\\
O(\sqrt{a})&O(1)&O(1/\sqrt{a})\\
O(a)&O(\sqrt{a})&O(\sqrt{a})\\
\end{bmatrix}.
\]
\end{lem}
\begin{proof} With ${\mathcal A}=({\tilde a}_{ij})$ it is clear that
\begin{equation}
\label{eq:nami:a}
{\tilde a}_{ij}=t_{1i}\,a\,t_{2j}+t_{1i}\,b\,t_{3j}+t_{2i}t_{1j}+t_{3i}t_{2j}.
\end{equation}
Since $t_{2i}t_{1j}=O(\sqrt{a})$ unless $(i,j)=(2,3)$ and $t_{3i}t_{2j}=O(\sqrt{a})$ unless $(i,j)=(1,2)$ the first assertion follows from \eqref{eq:ordT} and \eqref{eq:nami:a}. Note that $\dif_x\big(1/d_j\big)=O(1/a^{3/2})$ for $j=1,2$ and $\dif_x\big(1/d_3\big)=O(1)$ then the second assertion follows from \eqref{eq:dx:TtT:a}, \eqref{eq:nami:a} and \eqref{eq:ordT}.
\end{proof}
From Lemma \ref{lem:difTAT} and \eqref{eq:dx:TtT} one has
\[
{\mathcal A}(\dif_xT^{-1})T=\begin{bmatrix}O(1)&O(\sqrt{a})&O(a)\\
O(\sqrt{a})&O(a)&O(a^{3/2})\\
O(a)&O(a^{3/2})&O(a^2)\\
\end{bmatrix}.
\]
Then it is clear that $\big|\lr{\Lambda{\mathcal A}(\dif_xT^{-1})TV, V}\big|$ is bounded by
\[
a^{3/2}|V_1||V_2|+a|V_1||V_3|
+a^{3/2}\, |V_2||V_3|+\lambda_1|V_1|^2+a^2(|V_2|^2+|V_3|^2)
\]
where
\begin{gather*}
a^{3/2}|V_1||V_2|\preceq a\,\phi^{-1}\lambda_1|V_1|^2+\frac{a^2\phi}{\lambda_1}|V_2|^2\preceq a\,\phi^{-1}(\lambda_1|V_1|^2+a|V_2|^2),\\
a\,|V_1||V_3|\preceq a\,\phi^{-1}\lambda_1|V_1|^2+a\,\frac{\phi}{\lambda_1}|V_3|^2\preceq a\,\phi^{-1}(\lambda_1|V_1|^2+|V_3|^2),\\
a^{3/2}\,|V_2||V_3|\preceq a\,(a|V_2|^2+|V_3|^2)
\end{gather*}
hence $\big|\phi^{-N}\lr{\Lambda{\mathcal A}(\dif_xT^{-1})TV, V}\big|$ is bounded by $N\phi^{-N-1}\lr{\Lambda V,V}$ taking $N$ large.

%%%%%%%%%%%%%%%%%%%%%%
\subsection{Estimate of energy, term $\lr{\dif_x(\Lambda{\mathcal A})V,V}$}

Write 
\[
\lr{\dif_x(\Lambda{\mathcal A})V,V}=\lr{(\dif_x\Lambda){\mathcal A}V,V}+\lr{\Lambda(\dif_x{\mathcal A})V,V}
\]
and estimate each term on the right-hand side. To estimate the first term note
\begin{lem}
\label{lem:difx:lam}One has 
\begin{equation}
\label{eq:dla1:la1}
\big|\dif_x\lambda_1\big|/\lambda_1\preceq 1/\sqrt{a}+1/\sqrt{\Delta}\preceq 1/(\phi\sqrt{a}).
\end{equation}
\end{lem}
\begin{proof}Recall
\[
\dif_x\lambda_1=\big\{\dif_x(2a+a^2)\lambda_1^2-\dif_x(6a+2a^2+2a^3-9b^2)\lambda_1+\dif_x\Delta\big\}/q_{\lambda}(\lambda_1)
\]
where Lemma \ref{lem:NiP} shows that
\[
\frac{|\dif_x\lambda_1|}{\lambda_1}\preceq \frac{\sqrt{a}\,\lambda_1^2+\sqrt{a}\,\lambda_1}{\lambda_1\,a}+\frac{|\dif_x\Delta|}{\lambda_1\,a}\preceq \frac{1}{\sqrt{a}}+\frac{|\dif_x\Delta|}{\lambda_1\,a}\preceq \frac{1}{\sqrt{a}}+\frac{|\dif_x\Delta|}{\Delta}
\]
because $1/\lambda_1\preceq a/\Delta$ by Proposition \ref{pro:Skon}.  Noting that $\Delta\geq 0$ in a neighborhood of $x=0$ we see that $|\dif_x\Delta|\preceq \sqrt{\Delta}$, hence the first inequality. The second inequality follows from the first one thanks to the assumption \eqref{eq:katei:D}.
\end{proof}
Since $|\dif_x\lambda_2|/\lambda_2=O(1/\sqrt{a})$ and $|\dif_x\lambda_3|/\lambda_3=O(1)$ by  Lemma \ref{lem:bibun} it follows from  \ref{lem:difx:lam} that $\phi(\dif_x\Lambda)=\phi \Lambda(\Lambda^{-1}\dif_x\Lambda)={\rm diag}\big(O(\lambda_1/\sqrt{a}),O(\phi\sqrt{a}),O(\phi)\big)$ for $\lambda_2\simeq a$ then by Lemma \ref{lem:difTAT} and \eqref{eq:dx:TtT:a} one sees
\[
\phi(\dif_x\Lambda){\mathcal A}=\begin{bmatrix}
O(\lambda_1)&O(\lambda_1/\sqrt{a})&O(\lambda_1)\\
O(\phi\,a^{3/2})&O(\phi\,a)&O(\phi\,\sqrt{a})\\
O(\phi\, a^{3/2})&O(\phi\, a)&O(\phi\, a^{5/2})
\end{bmatrix}.
\]
Therefore to estimate $\phi\lr{(\dif_x\Lambda){\mathcal A}V,V}$ it suffices to bound
\begin{align*}
\Big(\frac{\lambda_1}{\sqrt{a}}+a^{3/2}\phi\Big)|V_1||V_2|+(\lambda_1+a^{3/2}\phi)|V_1||V_3|+\phi \sqrt{a}|V_2||V_3|.
\end{align*}
Since $\lambda_1\preceq a\sqrt{\lambda_1}$ and $\phi\preceq \sqrt{\lambda_1}$ this is bounded by $\lr{\Lambda V,V}$ 
and hence
\begin{equation}
\label{eq:DdifD:A}
\big|\phi^{-N}\lr{(\Lambda^{-1}\dif_x\Lambda){\mathcal A}V,\Lambda V}\big|\preceq \phi^{-N-1}\lr{\Lambda V,V}.
\end{equation}
As for $\big|\phi^{-N}\lr{(\dif_x{\mathcal A})V,\Lambda V}\big|$ it follows  from Lemma \ref{lem:difTAT} that
\[
\Lambda(\dif_x{\mathcal A})=\begin{bmatrix}
O(\lambda_1)&O(\lambda_1/\sqrt{a})&O(\lambda_1)\\
O(a^{3/2})&O(a)&O(\sqrt{a})\\
O(a)&O(\sqrt{a})&O(\sqrt{a})
\end{bmatrix}
\]
hence repeating similar arguments it is easy to see that
\begin{equation}
\label{eq:DdifA}
\big|\phi^{-N}\lr{\Lambda(\dif_x{\mathcal A})V,V}\big|\preceq \phi^{-N-1}\lr{\Lambda V,V}
\end{equation}
which is  bounded by $N\phi^{-N-1}\lr{\Lambda V,V}$ taking $N$ large.  
%%%%%%%%
  
%%%%%%%%

%
\begin{remark}
\rm It shoud be remarked that the condition \eqref{eq:katei:D},  assumed in this section, are stated in terms of $\Delta$ and $a$, where $a$ is constant times the discriminant of $\dif p/\dif\tau$,  without any reference to characteristic roots $\tau_i$. 
\end{remark}

We conclude this section with an important remark. To obtain energy estimates it  suffices to find a finite number of pairs $(\phi_j,\omega_j)$, where $\phi_j$ is a scalar weight  satisfying   \eqref{eq:katei:D} in subregion $\omega_j$ of which union covers a neighborhood of $(0,0)$,  such that one can {\it collect} such estimates obtained in $\omega_j$. In the following sections this observation is carried out for hyperbolic operators with effectively hyperbolic critical points with coefficients depending on $t$, also for third order hyperbolic operators with effectively hyperbolic critical points with two independent variables.

%%%%%%%%%%%
%%%%%%%%%%%%
\section{Hyperbolic operators with effectively hyperbolic critical points with time dependent coefficients}
\label{sec:app:tdep}

Ivrii and Petkov proved that if the Cauchy problem \eqref{eq:CPm} is $C^{\infty}$ well posed for any lower order term then the Hamilton map $H_p$ has a pair of  non-zero real eigenvalues at every critical point (\cite[Theorem 3]{IP}). 
Here the Hamilton map $F_p$ is defined by
\[
F_p(X,\Xi)=\begin{bmatrix}\displaystyle{\frac{\dif^2 p}{\dif X\dif\Xi}}&\displaystyle{\frac{\dif^2 p}{\dif \Xi\dif\Xi}}\\[10pt]
\displaystyle{-\frac{\dif^2 p}{\dif X\dif X}}&\displaystyle{-\frac{\dif^2 p}{\dif \Xi\dif X}}
\end{bmatrix},\quad X=(t,x),\;\;\Xi=(\tau,\xi)
\]
and $(X,\Xi)$ is a critical point if $\dif p/\dif X=\dif p/\dif \Xi=0$ at $(X,\Xi)$. Note that $p(X,\Xi)=0$ at critical points by the homogeneity in $\xi$ so that $(X,\Xi)$ is a multiple characteristic and $\tau$ is a multiple characteristic root. A critical point where the Hamilton map $H_p$ has a pair of non-zero real eigenvalues  is called {\it effectively hyperbolic} (\cite{Ho1}).  
In \cite{I}, Ivrii has proved that if  every critical point is effectively hyperbolic, and  $p$ admits a decomposition $p=q_1q_2$ with real smooth symbols  $q_i$ near the critical point then the Cauchy problem  is $C^{\infty}$ well-posed for every lower order term.  In this case the critical point is effectively hyperbolic  if and only if the Poisson bracket $\{q_1,q_2\}$ does not vanish. He has conjectured that the assertion would hold without any additional condition. 

If a critical point $(X,\Xi)$ is effectively hyperbolic then $\tau$ is a characteristic root of multiplicity at most $3$ (\cite[Lemma 8.1]{IP}). If every multiple characteristic root is at most double, the conjecture has been proved in \cite{I}, \cite{Mel}, \cite{Iwa1}, \cite{Iwa2}, \cite{Iwa3}, \cite{N4}, \cite{N5}.
For more details  about the conjecture and subsequent progress on several questions including the above conjecture, see \cite{BBP}, \cite{BeNi}, \cite{Ni:book}, \cite{NP}. 

Any characterisitc root $\tau$ of multiplicity $r\geq 4$ at $(t,x,\xi)$ with $t\geq 0$ and any characterisitc root $\tau$ of multiplicity $r\geq 3$ at $(t,x,\xi)$ with $t> 0$, the point $(t,x,\tau,\xi)$ is a critical point with $F_p(t,x,\tau,\xi) = O$. While any  characteristic root $\tau$ of multiplicity $3$ at $(0,x,\xi)$, the point $(0,x,\tau,\xi)$ is a critical point which may be effectively hyperbolic (\cite[Lemma 8.1]{IP}). Any double characteristic root $\tau$ at $(t,x,\xi)$ with $t>0$, the point $(t,x,\tau,\xi)$ is also a critical point which may be effectively hyperbolic while for double characteristic roots $\tau$ at  $(0,x,\xi)$, the characteristic points $(0,x,\tau,\xi)$  are not necessarily critical points. Here is a simple example in $\R^2$, $x\in\R$ and $t\geq 0$,
\[
P=(D_t^2-t^{\ell}D_x^2)(D_t+c\,D_x),\qquad \ell \in \N
\]
where $c\in\R$. Let $c\neq 0$ then it is clear that $\tau=0$ is a double characteristic root at $(0,0,1)$.  If $\ell=1$ then $\dif_tp(0,0,0,1)=-c\neq 0$ and hence $(0,0,0,1)$ is not a critical point. If $\ell\geq 2$ then $(0,0,0,1)$ is a critical point and $F_p$ has non-zero real eigenvalues there if and only if $\ell=2$. Let $c=0$ then $\tau=0$ is a triple characteristic root at $(0,0,1)$ so that $(0,0,0,1)$ is a critical point and $F_p$ has non-zero real eigenvalues there if and only if $\ell= 1$.

%%%%%%%%%%%%

Now we restrict ourselves to the case that the coefficients depend only on $t$ and consider the Cauchy problem 
\begin{equation}
\label{eq:tdepCPm}
\begin{cases}\displaystyle{Pu=D_t^mu+\sum_{j=0}^{m-1}\sum_{j+|\alpha|\leq m}a_{j,\alpha}(t)D_x^{\alpha}D_t^ju=0},\;(t,x)\in [0,T)\times \R^n\\
D_t^ku(0,x)=u_k(x),\;\;x\in\R^n,\quad k=0,\ldots,m-1
\end{cases}
\end{equation}
where $a_{j,\alpha}(t)$ ($j+|\alpha|=m$) are real valued and  $C^{\infty}$ in $(-c,T)$ with some $c>0$ and the principal symbol $p$ is {\it hyperbolic for $t\geq 0$}, that is
\begin{equation}
\label{eq:ordm:p}
p(t,\tau,\xi)=\tau^m+\sum_{j=0}^{m-1}\sum_{j+|\alpha|=m}a_{j,\alpha}(t)\xi^{\alpha}\tau^j
\end{equation}
has only real roots in $\tau$ for any $(t,\xi)\in [0,T)\times \R^n$.  
\begin{theorem}
\label{thm:tdep:kyosoku}If every  critical point $(0,\tau,\xi)$,  $\xi\neq0$ is effectively hyperbolic then there exists $\delta>0$ such that for any $a_{j,\alpha}(t)$ with $j+|\alpha|\leq m-1$, which are $C^{\infty}$ in a neighborhood of $[0,\delta]$, the Cauchy problem \eqref{eq:tdepCPm} with $T=\delta$ is $C^{\infty}$ well-posed.
\end{theorem}
\begin{remark}
\label{rem:tfu} \rm Under the assumption of Theorem \ref{thm:tdep:kyosoku} it follows that $p$ has necessarily {\it non-real} characteristic roots in the $t<0$ side near $(0,\xi)$, that is $P$ would be a {\it Tricomi type} operator.  Indeed from \cite[Lemma 8.1]{IP} it follows  that $F_p(0,\tau,\xi)=O$ if all characteristic roots are real in a {\it full} neighborhood of $(0,\xi)$. 
\end{remark}
%

%%%%%%%%%%%%%%%%%%%%
\subsection{Triple effectively hyperbolic critical points}
\label{sec:tdep}

Assume that $p(t,\tau,\xi)$ has a triple characteristic root  ${\bar\tau}$ at $(0,{\bar\xi})$, $|{\bar\xi}|=1$ and $(0,{\bar\tau},{\bar\xi})$ is effectively hyperbolic. As we see later, without restrictions one may assume that $m=3$ and 
$p$ has the form
\begin{equation}
\label{eq:sanji}
p(t,\tau,\xi)=\tau^3-a(t,\xi)|\xi|^2\tau-b(t,\xi)|\xi|^3
\end{equation}
where $a(t,\xi)$ and $b(t,\xi)$ are homogeneous of degree $0$ in $\xi$ and satisfy
\begin{equation}
\label{eq:cond:t:1}
\Delta(t,\xi)=4\,a(t,\xi)^3-27\,b(t,\xi)^2\geq 0, \quad (t,\xi) \in   [0,T) \times \R^n. 
\end{equation}
 The triple characteristic root of $p(0,\tau,{\bar\xi})=0$ is  $\tau=0$ and 
\[
{\rm det}\,\big(\lambda-F_p(0,0,{\bar\xi})\big)=\lambda^{2n}\big(\lambda^2-\{\dif_ta(0,{\bar\xi})\}^2\big)
\]
hence $(0,0,{\bar\xi})$ is effectively hyperbolic if and only if
\begin{equation}
\label{eq:difa}
\dif_t a(0,{\bar\xi})\neq 0.
\end{equation}
Since $a(0,{\bar\xi})=0$ and $\dif_ta(0,{\bar\xi})\neq 0$  there is a neighborhood ${\mathcal U}$ of $(0,{\bar\xi})$ in which one can write
\begin{equation}
\label{eq:Malg:a}
a(t,\xi)=e_1(t,\xi)(t+\alpha(\xi))
\end{equation}
where $e_1>0$ in ${\mathcal U}$. Note that $\alpha(\xi)\geq 0$ near ${\bar\xi}$ because $a(t,\xi)\geq 0$ in $[0,T)\times \R^n$.
\begin{lem}
\label{lem:DMal}There exists a neighborhood ${\mathcal U}$ of $(0,{\bar\xi})$ in which one can write
\[
\Delta(t,\xi)=e_2(t,\xi)\big\{t^3+a_1(\xi)t^2+a_2(\xi)t+a_3(\xi)\big\}
\]
where $e_2>0$ and $a_j({\bar\xi})=0$.
\end{lem}
\begin{proof}
Thanks to the Malgrange preparation theorem it suffices to show 
\[
\dif_t^k\Delta(0,{\bar\xi})=0,\;\;k=0,1,2,\quad \dif_t^3\Delta(0,{\bar\xi})\neq 0.
\]
It is clear that $\dif_t^ka^3=0$ at $(0,{\bar\xi})$ for $k=0,1,2$ and $\dif_t^3a(0,{\bar\xi})\neq 0$. Since $\Delta=4a^3-27b^2$ and $b(0,{\bar\xi})=0$ it is enough to show $\dif_tb(0,{\bar\xi})=0$. Suppose $\dif_tb(0,{\bar\xi})\neq0$ and hence
\[
b(t,{\bar\xi})=t\big(b_1+tb_2(t)\big)
\]
where $b_1\neq0$. Since $a(t,{\bar\xi})=c\,t$ with $c>0$ then $\Delta(t,{\bar\xi})=4\,c^3\,t^3-27b(t,{\bar\xi})^2\geq 0$ is impossible. This proves the assertion.
\end{proof}
\begin{lem}
\label{lem:alnu}There exist a neighborhood $U$ of ${\bar\xi}$ and a positive constant $\varep>0$ such that for any $\xi\in U$ one can find $j\in\{1,2,3\}$ such that
\[
\varep^{-1}\,|\nu_j(\xi)|\geq \alpha(\xi)
\]
where $\nu_j(\xi)$ are roots of $
t^3+a_1(\xi)t^2+a_2(\xi)t+a_3(\xi)=0$.
\end{lem}
\begin{proof}In view of Lemma \ref{lem:DMal} one can write
\begin{align*}
\Delta=4a^3-27b^2=4e_1^3(t+\alpha)^3-27b^3=4e_1^3\big\{(t+\alpha)^3-{\hat b}^2\big\}\\
=e_2\big\{t^3+a_1(\xi)t^2+a_2(\xi)t+a_3(\xi)\big\}
\end{align*}
where ${\hat b}=3\sqrt{3}\,b/(2e_1^{3/2})$ and hence
\[
(t+\alpha)^3-{\hat b}^2=E\big\{t^3+a_1(\xi)t^2+a_2(\xi)t+a_3(\xi)\big\}
\]
with $E=e_2/(4e_1^3)$. Choose $I=[-\delta_1,\delta_1]$, $\delta_1>0$ and a neighborhood $U_1$ of ${\bar\xi}$ such that $I\times U_1\subset {\mathcal U}$ and denote
\[
\sup_{(t,\xi)\in I\times U_1, 0\leq k\leq 2}\big|\dif_t^kE(t,\xi)\big|=C.
\]
Write $
{\hat b}(t,\xi)={\hat b}_0(\xi)+{\hat b}_1(\xi)t+{\hat b}_2(\xi)t^2+{\hat b}_3(t,\xi)t^3$ 
and set
\begin{equation}
\label{eq:bhat}
\sup_{\xi\in U_1}\alpha(\xi)+\sup_{\xi\in U_1,0\leq k\leq 2}|{\hat b}_k(\xi)|+\sup_{(t,\xi)\in I\times U_1}|{\hat b}_3(t,\xi)|=B.
\end{equation}
Choose a neighborhood $U\subset U_1$ of ${\bar\xi}$ such that
\[
4\,B\sqrt{\alpha(\xi)}<1\quad \text{for}\;\;\xi\in U
\]
which is possible because $\alpha({\bar\xi})=0$. Choose $\varep=\varep(B,C)>0$ such that
\begin{equation}
\label{eq:erabu}
6\big(1-\varep C(1+\varep B+\varep^2B^2/6)\big)\\
-\frac{9}{2}\big(1+\varep^2C(1+\varep B)\big)^2/(1-\varep^3 C)>1
\end{equation}
and prove that if there exists $\xi\in U$ such that
\begin{equation}
\label{eq:uso}
\big|\nu_j(\xi)\big|<\epsilon\,\alpha(\xi),\quad j=1,2,3
\end{equation}
we would have a contradiction. We omit to write $\xi$ for simplicity. Recall
\begin{equation}
\label{eq:nu:1}
(t+\alpha)^3-{\hat b}^2=E\prod(t-\nu_j)\geq 0\quad (t\geq 0)
\end{equation}
and hence, taking $t=0$ one has $\alpha^3-{\hat b}_0^2=E(0)|\nu_1\nu_2\nu_3|<C\varep^3\alpha^3$. 
This shows
\begin{equation}
\label{eq:nu:2}
\sqrt{1-C\varep^3}\,\alpha^{3/2}\leq |{\hat b}_0|\leq \alpha^{3/2}.
\end{equation}
Differentiating \eqref{eq:nu:1} by $t$ and putting $t=0$ one has
\[
|3\alpha^2-2{\hat b}_0{\hat b}_1|\leq C\varep^3\alpha^3+3C\varep^2\alpha^2.
\]
This gives 
\[
3\alpha^2\big(1-C\varep^2(1+\varep B)\big)\leq 2|{\hat b}_0{\hat b}_1|\leq 3\alpha^2\big(1+C\varep^2(1+\varep B)\big).
\]
In view of \eqref{eq:nu:2} one has
\begin{equation}
\label{eq:nu:3}
\frac{3}{2}\alpha^{1/2}\big(1-C\varep^2(1+\varep B)\big)\leq |{\hat b}_1|\leq \frac{3}{2}\frac{\alpha^{1/2}}{\sqrt{1-C\varep^3}}\big(1+C\varep^2(1+\varep B)\big).
\end{equation}
Differentiating \eqref{eq:nu:1} twice by $t$ and putting $t=0$ on has
\[
\big|6\alpha-(2{\hat b}_1^2+4{\hat b}_0{\hat b}_2)\big|\leq C\,\varep\,\alpha(6+6\varep\,B+\varep^2\,B^2)
\]
which proves $
\big|4{\hat b}_0{\hat b}_2+2{\hat b}_1^2\big|\geq 6\,\alpha\big(1-\varep C-\varep^2CB-\varep^3CB^2/6\big)$. 
Using \eqref{eq:nu:3} one obtains
\begin{align*}
|4{\hat b}_0{\hat b}_2|\geq 6\alpha\big(1-\varep C(1+\varep B+\varep^2B^2/6)\big)\\
-\frac{9}{2}\alpha\big(1+\varep^2C(1+\varep B)\big)^2/(1-\varep^3 C)
\end{align*}
where the right-hand side is greater than $\alpha$ by \eqref{eq:erabu}. On the other hand from \eqref{eq:bhat} and  \eqref{eq:nu:2} we have 
\[
4B\,\alpha^{3/2}\geq 4\,\alpha^{3/2}\,|{\hat b}_2|\geq  4\,|{\hat b}_0{\hat b}_1|> \alpha
\]
and hence $4\,B\sqrt{\alpha}> 1$ which contradicts with
\eqref{eq:uso}.
\end{proof}
Denote $\Delta/e_2$ by ${\bar \Delta}$;
\[
{\bar \Delta}(t,\xi)=\Delta/e_2=t^3+a_1(\xi)t^2+a_2(\xi)t+a_3(\xi).
\]
\begin{lem}
\label{lem:Dkon}There is a neighborhood  $V\subset U$ of ${\bar\xi}$ where
one can write either
\begin{equation}
\label{eq:ba:1}
{\bar\Delta}=\big|t-\nu_2(\xi)\big|^2\big(t-\nu_1(\xi)\big),\quad \nu_1(\xi)\;\;\text{is real and}\;\; \nu_1(\xi)\leq 0 
\end{equation}
or
\begin{equation}
\label{eq:ba:2}
{\bar\Delta}=\prod_{k=1}^3(t-\nu_k(\xi)\big),\quad \nu_k(\xi)\;\;\text{are real and}\;\; \nu_k(\xi)\leq 0.
\end{equation}
\end{lem}
\begin{proof}Let $\nu_j(\xi)$, $j=1,2,3$ be the roots of ${\bar \Delta}(t,\xi)=0$. Since $\nu_j({\bar\xi})=0$ one can assume $|\nu_j(\xi)|<\delta_1$ in $V$. Since $a_j(\xi)$ are real we have two cases; one is real and other two are complex conjugate or all three are real. For the former case denoting the real root by $\nu_1(\xi)$ we have \eqref{eq:ba:1} where $\nu_1(\xi)\leq 0$ because ${\bar \Delta}\geq 0$ for $0\leq t\leq \delta_1$. In the latter case, if two of them coincide, denoting the remaining one by $\nu_1(\xi)$ one has \eqref{eq:ba:1}. If $\nu_j(\xi)$ are different each other then we have \eqref{eq:ba:2}  since ${\bar\Delta}\geq 0$ for $0\leq t\leq \delta_1$.
\end{proof}
%

%%%%%%%%%%%%%%%%
\subsection{Key proposition}

Thanks to Lemma \ref{lem:Dkon} we have either \eqref{eq:ba:1} or \eqref{eq:ba:2}. As was observed in \cite{N1} (see also \cite{N3}) in order to obtain energy estimates it is important to consider not only real zeros of $\Delta$ but also ${\mathsf{Re}}\,\nu_j(\xi)$, the real part of $\nu_j(\xi)$ for non-real zeros. Define
\begin{align*}
&\psi(\xi)=\max{\big\{0,{\mathsf{Re}}\,\nu_j(\xi)\big\}}=\max{\big\{0,{\mathsf{Re}}\,\nu_2(\xi)\big\}},\\
&\phi_1=t, \;\;\omega_1(\xi)=[0,\psi(\xi)/2],\quad
\phi_2=t-\psi(\xi),\;\;\omega_2(\xi)=[\psi(\xi)/2,\delta] 
\end{align*}
with small $\delta>0$. 
If $\psi(\xi)\leq 0$,  we have $\phi_1=\phi_2=t$ and $\omega_2=[0,\delta]$.
The next proposition is the key to applying the arguments in Section \ref{sec:enehyoka:1} to operators with triple effectively hyperbolic characteristics.
\begin{prop}
\label{pro:kihon}There exist a neighborhood $U$ of  ${\bar\xi}$, positive constants $\delta>0$ and $C>0$ such that 
\begin{equation}
\label{eq:kihon:a}
\phi_j^2\,a\leq C\Delta,\quad |\phi_j|\,|\dif_t\Delta|\leq C\Delta,\quad |\phi_j|\leq C\,a
\end{equation}
for any $\xi\in U$ and $t\in \omega_j(\xi)$, $j=1,2$.
\end{prop}
\begin{proof} Thanks to \eqref{eq:Malg:a} and Lemma \ref{lem:DMal} it suffices to prove  \eqref{eq:kihon:a} for ${\bar \Delta}$ and $t+\alpha$ instead of $\Delta$ and $a$. 
Note that
\begin{align*}
&\Big|\frac{\dif_t{\bar\Delta}}{\bar\Delta}\Big|\leq\frac{1}{t+|\nu_1(\xi)|}+2\,\frac{|t-{\mathsf{Re}}\,\nu_2(\xi)|}{|t-\nu_2(\xi)|^2}\leq \frac{1}{\phi_1}+\frac{2}{|\phi_2|},\\
&\Big|\frac{\dif_t{\bar\Delta}}{\bar\Delta}\Big|\leq\sum_{k=1}^3\frac{1}{t+|\nu_k(\xi)|}\leq \frac{3}{\phi_1}
\end{align*}
in the case \eqref{eq:ba:1} and \eqref{eq:ba:2} respectively. Since $|\phi_2|\geq t=\phi_1$ in $\omega_1(\xi)$ and $t=\phi_1\geq |\phi_2|$ in $\omega_2(\xi)$ it is easy to see 
\[
|\phi_j|\,|\dif_t{\bar\Delta}|\leq 3\,{\bar \Delta}\quad \text{in}\;\;\omega_j(\xi)
\]
for both cases. Similarly noting  $t+\alpha(\xi)\geq t$ it is clear that
\[
|\phi_j|\leq t+\alpha \quad \text{in}\;\;\omega_j(\xi).
\]
Therefore it rests to prove $\phi_j^2\,(t+\alpha)\leq C{\bar\Delta}$ in $\omega_j(\xi)$. First we study the case \eqref{eq:ba:1}. From Lemma \ref{lem:alnu} either $\varep^{-1}|\nu_1|\geq \alpha $ or $\varep^{-1}|\nu_2|\geq \alpha$ holds. First assume that $\varep^{-1}|\nu_1|\geq \alpha$  and hence $t+|\nu_1|\geq \varep(t+\alpha)$ then
\[
\frac{{\bar\Delta}}{|\phi_j|(t+\alpha)}\geq \varep\,\frac{|t-\nu_2|^2}{|\phi_j|}\geq \varep\,\frac{|t-{\mathsf{Re}}\,\nu_2|^2}{|\phi_j|}=\varep\,\frac{\phi_2^2}{|\phi_j|}\geq \varep |\phi_j|,\quad t\in\omega_j.
\]
Next assume $\varep^{-1}|\nu_2|\geq \alpha$. If $0<{\mathsf{Re}}\,\nu_2\leq |{\mathsf{Im}}\,\nu_2|$ one has for $t\geq 0$ 
\begin{align*}
&| t-{\mathsf{Re}}\,\nu_2|+|{\mathsf{Im}}\,\nu_2|\geq t-{\mathsf{Re}}\,\nu_2+|{\mathsf{Im}}\,\nu_2|\geq t,\\
&| t-{\mathsf{Re}}\,\nu_2|+|{\mathsf{Im}}\,\nu_2|\geq |{\mathsf{Im}}\,\nu_2|\geq  |\nu_2|/2\geq \frac{\varep}{2}\alpha
\end{align*}
which also holds for ${\mathsf{Re}}\,\nu_2\leq 0$ clearly. Then we see that
\[
|t-\nu_2|\geq \frac{1}{2}\big(| t-{\mathsf{Re}}\,\nu_2|+|{\mathsf{Im}}\,\nu_2|\big)\geq \frac{\varep}{2(2+\varep)}(t+\alpha).
\]
Therefore it follows that
\[
\frac{{\bar\Delta}}{|\phi_j|(t+\alpha)}\geq \frac{\varep}{2(2+\varep)}\,\frac{t\,|t-\nu_2|}{|\phi_j|}\geq \frac{\varep}{2(2+\varep)}\,\frac{|\phi_1\phi_2|}{|\phi_j|}\geq \frac{\varep}{2(2+\varep)}|\phi_j|,\quad t\in\omega_j.
\]
If ${\mathsf{Re}}\,\nu_2>|{\mathsf{Im}}\,\nu_2|$ noting that, for $ t\in\omega_1$ 
\[
|t-{\mathsf{Re}}\,\nu_2|\geq {\mathsf{Re}}\,\nu_2/2\geq |\nu_2|/4\geq \varep\alpha/4,\quad |t-{\mathsf{Re}}\,\nu_2|\geq t
\]
one has $(4+\varep)|t-{\mathsf{Re}}\,\nu_2|\geq \varep(t+\alpha)$ in $\omega_1$. Hence
\[
\frac{{\bar\Delta}}{t(t+\alpha)}\geq \frac{|t-{\mathsf{Re}}\,\nu_2|^2}{t+\alpha}\geq \frac{\varep}{4+\varep}\,|t-{\mathsf{Re}}\,\nu_2|\geq \frac{\varep}{4+\varep}\,t,\quad t\in\omega_1.
\]
For $t\in\omega_2$ note that 
\[
t\geq {\mathsf{Re}}\,\nu_2/2\geq |\nu_2|/4\geq \varep\alpha/4
\]
and hence $(4+\varep)\,t\geq \varep(t+\alpha)$. Thus one has
\[
\frac{{\bar\Delta}}{|t-{\mathsf{Re}}\,\nu_2|(t+\alpha)}\geq \frac{t\,|t-{\mathsf{Re}}\,\nu_2|}{(t+\alpha)}\geq \frac{\varep}{4+\varep}|t-{\mathsf{Re}}\,\nu_2|,\quad t\in\omega_2
\]
which proves the assertion for the case \eqref{eq:ba:1}.  In the case \eqref{eq:ba:2} note that
\begin{align*}
{\bar\Delta}=\prod_{k=1}^3\big(t+|\nu_k|\big).
\end{align*}
If $\varep^{-1}|\nu_j|\geq \alpha$ then $t+|\nu_j|\geq \varep(t+\alpha)$ and hence it is clear that
\[
{\bar\Delta}\geq \varep\,t^2\,(t+\alpha)
\]
which shows the assertion. Thus the proof of Proposition \ref{pro:kihon} is completed.
\end{proof}
%

%%%%%%%%%%%%%%%%
\subsection{Energy estimates}
\label{sec:ene:tdep}

Let $P$ be a differential operator of order $3$ with coefficients depending on $t$. 
After Fourier transform in $x$ the equation $Pu=f$ reduces to
\begin{equation}
\label{eq:eqFT}
D_t^3{\hat u}+\sum_{j+|\alpha|\leq 3, j\leq 2}a_{j,\alpha}(t)\xi^{\alpha}D_t^j{\hat u}={\hat f}
\end{equation}
where ${\hat u}(t,\xi)$ stands for the Fourier transform of $u(t,x)$ with respect to $x$. With
\[
E(t,\xi)=\exp{\Big(\frac{i}{3}\int_0^t \sum_{|\alpha|=1}a_{2,\alpha}(s)\xi^{\alpha}\,ds\Big)}
\]
it is clear that  ${\hat v}=E(t,\xi){\hat u}$ satisfies
\begin{equation}
\label{eq:eqFT:b}
D_t^3{\hat v}-a(t,\xi)|\xi|^2D_t{\hat v}-b(t,\xi)|\xi|^3{\hat v}
+\sum_{j=1}^3b_j(t,\xi)|\xi|^{j-1}D_t^{3-j}{\hat v}=E{\hat f}
\end{equation}
where $b_j(t,\xi)=0$ for $|\xi|\leq 1$ can be assumed since energy estimates for $|\xi|\leq 1$ is easily obtained. Since
\[
\sum_{k=0}^{\ell}\big|(|\xi|+1)^{\ell-k}\dif_t^k{\hat u}(t)\big|^2\leq C_{\ell}\sum_{k=0}^{\ell}\big|(|\xi|+1)^{\ell-k}\dif_t^k{\hat v}(t)\big|^2
\]
in order to obtain energy estimates for ${\hat u}$ one can assume that ${\hat u}$ satisfies \eqref{eq:eqFT:b} from the beginning. With $U={^t}\big(D_t^2{\hat u},|\xi|D_t{\hat u}, |\xi|^2{\hat u}\big)$ the equation \eqref{eq:eqFT:b} can be written
\begin{equation}
\label{eq:eqU}
\begin{split}
\frac{\dif}{\dif t}U&=i\begin{bmatrix}0&a(t,\xi)&b(t,\xi)\\
                                    1&0&0\\
                                    0&1&0\\
                                    \end{bmatrix}\!|\xi|\,U\\[3pt]
                                    &+i\begin{bmatrix}
                                    b_1(t,\xi)&b_2(t,\xi)&b_3(t,\xi)\\
                                    0&0&0\\
                                    0&0&0\\
                                    \end{bmatrix}U+\begin{bmatrix}iE{\hat f}\;\\
                                    0\\
                                    0\\
                                    \end{bmatrix}
                                   \\[5pt]
                                   &=iA|\xi|U+BU+F.
                                    \end{split}
\end{equation}
Let $S(t,\xi)$ and $T(t,\xi)$ be defined in Section \ref{sec:DiaSym} with $X=\xi$ such that $T^{-1}ST=\Lambda={\rm diag}\,(\lambda_1,\lambda_2,\lambda_3)$.   With $V=T^{-1}U$ one has
\begin{align*}
\dif_tV&=iA^T|\xi|V+\big(B^T+(\dif_tT^{-1})T\big)V+T^{-1}F\\
&=i{\mathcal A}|\xi|V+{\mathcal B}V+{\tilde F}
\end{align*}
 where ${\mathcal A}=T^{-1}AT$ and ${\mathcal B}=T^{-1}BT+ (\dif_tT^{-1})T$.    
 Thanks to Proposition \ref{pro:kihon} we have candidates for scalar weights in each $\omega_j$. To simplify notation denote
\[
t_0(\xi)=0,\quad t_1(\xi)=\psi(\xi)/2,\quad t_2(\xi)=\psi(\xi),\quad t_3(\xi)=\delta
\]
and following  \cite{N1} (also \cite{N3})  introduce three subintervals $\Omega_j=[t_{j-1}(\xi),t_j(\xi)]$ and scalar weights $\varphi_j$, $j=1,2,3$
\begin{align*}
\varphi_1(t,\xi)=t,\quad \varphi_2(t,\xi)=\psi(\xi)-t,\quad \varphi_3(\xi)=t-\psi(\xi).
\end{align*}
Note that $\omega_1=\Omega_1$, $\omega_2=\Omega_2\cup\Omega_3$ and $\varphi_j=|\phi_2|$ in $\Omega_j$, $j=2,3$. Thanks to Proposition \ref{pro:kihon} and Lemma \ref{lem:katei:D} one has
 \begin{equation}
 \label{eq:katei:2b}
 \varphi_j^2\leq C\lambda_1,\;\;\varphi_j|\dif_t\lambda_1|\leq C\lambda_1,\;\;\varphi_j\leq C\lambda_2,\quad t\in \Omega_j(\xi),\quad j=1,2,3
 \end{equation}
 where $C$ is independent of $\xi\in U$. Consider the following  energy in $\Omega_j(\xi)$;
 \[
{\mathcal E}_j=g_j\lr{\Lambda V,V}=g_j\sum_{k=1}^3\lambda_k(t,\xi)|V_k(t,\xi)|^2,\quad g_j=\varphi_j(t,\xi)^{2(-1)^jN-1}.
 \]
Since ${\mathsf{Re}}\,\lr{i\Lambda{\mathcal A}|\xi|V,V}=0$ and $\dif_t\varphi_j=(-1)^{j-1}$ one has
 \begin{align*}
 \frac{d}{dt}{\mathcal E}_j=&-(2N-(-1)^j)\varphi_j^{-1}{\mathcal E}_j+g_j\lr{(\dif_t\Lambda)V,V}\\
 &+2g_j{\mathsf{Re}}\,\lr{\Lambda{\mathcal B}V,V} +2g_j{\mathsf{Re}}\,\lr{\Lambda{\tilde F},V}.
 \end{align*}
 Repeating the same arguments as in Section \ref{sec:enehyoka:1} one can estimate
 \[
 \big|g_j\lr{\Lambda(\dif_tT^{-1})TV,V}\big|+
\big|g_j\lr{(\dif_t\Lambda)V,V}\big|\leq C_1\varphi_j^{-1}{\mathcal E}_j
\]
 in $\Omega_j$. It rests to estimate $|g_j\lr{\Lambda B^TV,V}|$. 
Let $C'$ be a bound of all entries of $B^T$. Since $0\leq \lambda_1\leq \lambda_2\leq \lambda_3$ then $|g_j\lr{\Lambda B^TV,V}|$ is bounded by
\begin{align*}
C'g_j\sum_{k=1}^3\sum_{l=1}^3\lambda_k|V_l|\,|V_k|\leq C'g_j\sum_{k=1}^3\lambda_k|V_k|^2+2\,C'g_j\lambda_2|V_1||V_2|\\
+2\,C'g_j\lambda_3\big(|V_1||V_3|+|V_2||V_3|\big).
\end{align*}
Thanks to \eqref{eq:katei:2b} one has
\[
2\lambda_2|V_1||V_2|\leq \sqrt{\lambda_2}\varphi_j|V_1|^2+\sqrt{\lambda_2}\frac{\lambda_2}{\varphi_j}|V_2|^2
 \leq C\sqrt{\lambda_2}\,\varphi_j^{-1}\big(\lambda_1|V_1|^2+\lambda_2|V_2|^2\big)
\]
and
\begin{align*}
2\lambda_3|V_1||V_3|\leq \varphi_j|V_1|^2+\varphi_j^{-1}\lambda_3^2|V_3|^2\leq C\,\varphi_j^{-1}\big(\lambda_1|V_1|^2+\lambda_3|V_3|^2\big),\\
2\lambda_3|V_2||V_3|\leq \varphi_j^{1/2}|V_2|^2+\varphi_j^{-1/2}|V_3|^2\leq C\sqrt{\varphi_j}\varphi_j^{-1}\big(\lambda_2|V_2|^2+\lambda_3|V_3|^2\big)
\end{align*}
and therefore there exists $C_2$ such that
\begin{equation}
\label{eq:FtoD}
|g_j\lr{\Lambda B^TV,V}|\leq C_2\,\varphi_j^{-1}{\mathcal E}_j\quad\text{in}\;\;\Omega_j.
\end{equation}
Noting that 
\[
2\big|{\mathsf{Re}}\,\lr{\Lambda{\tilde F},V}\big|\leq \varphi_j\lr{\Lambda{\tilde F},{\tilde F}}+\varphi_j^{-1}\lr{\Lambda V,V}
\]
and $|\lr{\Lambda{\tilde F},{\tilde F}}|\leq C''\|{\tilde F}\|^2=C''\|F\|^2$ we obtain
\begin{lem}
\label{lem:kukan:1}Let $j=1$ or $3$. There exist $N_0$ and $C>0$ such that for any $N\geq N_0$ and any $U(t,\xi)$ verifying $\dif_t^kU(t_{j-1}(\xi),\xi)=0$, $k=0,1,\ldots, N$ one has
\begin{align*}
\|U(t)\|^2+N\int_{t_{j-1}}^t\varphi_j^{-2N}(s)\|U(s)\|^2ds\leq C\int_{t_{j-1}}^t\varphi_j^{-2N}(s)\|F(s)\|^2ds
\end{align*}
for $t\in \Omega_j(\xi)$.
\end{lem}
With $\lr{\xi}=|\xi|+1$ it follows from Lemma \ref{lem:kukan:1} that
\begin{cor}
\label{cor:kukan:1}Let $j=1$ or $3$. There is $N_0$ such that for any $L\in\N$ there exists $C_L>0$ such that for any $U$ with $\dif_t^kU(t_{j-1}(\xi),\xi)=0$, $k=0,1,\ldots,N+L$ one has
\begin{equation}
\label{eq:indL}
\begin{split}
\sum_{k=0}^L\left\|\lr{\xi}^{L-k}\dif_t^kU(t)\right\|^2+N\sum_{k=0}^L\int_{t_{j-1}}^t\varphi_j^{-2N}(s)\left\|\lr{\xi}^{L-k}\dif_t^kU(s)\right\|^2ds\\
\leq C_L\sum_{k=0}^L\int_{t_{j-1}}^t\varphi_j^{-2N}(s)\left\|\lr{\xi}^{L-k}\dif_t^kF(s)\right\|^2ds
\end{split}
\end{equation}
for $t\in \Omega_j(\xi)$ and $N\geq N_0$.
\end{cor}
For the subinterval $\Omega_2(\xi)$ the argument in Section \ref{sec:enehyoka:1} shows again
\begin{lem}
\label{lem:kukan:2}There exist $N_0\in\N$ and $C>0$ such that one has 
\begin{equation}
\label{eq:indL:2}
\begin{split}
\varphi_2^{2N-1}(t)\|U(t)\|^2+N\int_{t_1}^t\varphi_2^{2N}(s)\|U(s)\|^2ds\\
\leq C\|U(t)\|^2+C\int_{t_1}^t\varphi_2^{2N}(s)\|F(s)\|^2ds.
\end{split}
\end{equation}
for $t\in \Omega_2(\xi)$ and  $N\geq N_0$.
\end{lem}
\begin{cor}
\label{cor:kukan:2}There exists $N_0$ such that for any $L\in \N$ there is $C_L$ such that
\begin{equation}
\label{eq: indL:2bis} 
\begin{split}
\varphi_2^{2N-1}(t)\sum_{k=0}^L\left\|\lr{\xi}^{L-k}\dif_t^kU(t)\right\|^2+N\sum_{k=0}^L\int_{t_1}^t\varphi_2^{2N}(s)\left\|\lr{\xi}^{L-k}\dif_t^kU(s)\right\|^2ds\\
\leq C_L\sum_{k=0}^L\left\|\lr{\xi}^{L-k}\dif_t^kU(t_1)\right\|^2+C_L\sum_{k=0}^L\int_{t_1}^t\varphi_2^{2N}(s)\left\|\lr{\xi}^{L-k}\dif_t^kF(s)\right\|^2ds
\end{split}
\end{equation}
for $t\in\Omega_2$ and $N\geq N_0$.
\end{cor}
Since energy estimates in each subinterval $\Omega_j$ is obtained, repeating the same arguments as in \cite{N1}, \cite[Section 6]{N3} one can collect the energy estimates in $\Omega_j$ yielding energy estimates of $U(t,\xi)$ in the whole interval $[0,\delta]$.
\begin{prop}
\label{pro:tdep}Assume that $p$ has a triple characteristic root ${\bar \tau}$ at $(0,{\bar \xi})$, $|\bar{\xi}|=1$ and $(0,{\bar\tau},{\bar\xi})$ is effectively hyperbolic. Then there exist $\delta>0$ and a conic neighborhood $U$ of ${\bar\xi}$ such that for any $a_{j,\alpha}(t)$ with $j+|\alpha|\leq 2$ one can find $N_0\in \N$ such that for any $q\in \N$ with $q\geq N_0$ there is $C>0$  such that
\begin{equation}
\label{eq:matome:tdep}
\sum_{k=0}^{q+3}\big|\lr{\xi}^{q+2-k}\dif_t^k{\hat u}(t,\xi)\big|^2\leq C\sum_{k=0}^q\int_0^t\big|\lr{\xi}^{N_0+q-k}\dif_t^k{\hat f}(s,\xi)\big|^2ds
\end{equation}
for $(t,\xi)\in [0,\delta]\times U$ and for any ${\hat u}(t,\xi)$ with $\dif_t^k{\hat u}(0,\xi)=0$, $k=0,1,2$ and  ${\hat f}(t,\xi)$ with $\dif_t^k{\hat f}(0,\xi)=0$, $k=0,\ldots,q+N_0$ satisfying \eqref{eq:eqFT}.
\end{prop}
%

%%%%%%%%%%%%%%%%%%%%%
\subsection{Remarks on double characteristics}

Assume that $P$ is a differential operator of order $2$ and the principal symbol $p$ has a double characteristic root ${\bar\tau}$ at $(0,{\bar \xi})$, $|\xi|=1$. After Fourier transform in $x$ the equation $Pu=f$ reduces to
\begin{equation}
\label{eq:eqFT:2}
D_t^2{\hat u}+\sum_{j+|\alpha|\leq 2, j\leq 1}a_{j,\alpha}(t)\xi^{\alpha}D_t^j{\hat u}={\hat f}
\end{equation}
Making similar procedure in Section \ref{sec:ene:tdep} one can assume that the principal symbol $p$ has the form
\begin{equation}
\label{eq:p:double}
p(t,\tau,\xi)=\tau^2-a(t,\xi)|\xi|^2,\quad a(0,{\bar\xi})=0
\end{equation}
so that ${\bar \tau}=0$ is a double characteristic root. 

If $\dif_ta(0,{\bar\xi})\neq 0$ one can write
\[
a(t,\xi)=e_1(t,\xi)\big(t+\alpha(\xi)\big),\quad \alpha({\bar\xi})=0
\]
in some neighborhood ${\mathcal U}$ of $(0,{\bar \xi})$ where $e_1>0$ and $\alpha(\xi)\geq 0$ near ${\bar\xi}$. In this case we choose $
\varphi_1=t$, $\Omega_1=[0,\delta]$ so that $\varphi_2$ is not needed. 

If $(0,0,{\bar\xi})$ is a critical point (hence  $\dif_ta(0,{\bar\xi})= 0$) and effectively hyperbolic then 
\begin{equation}
\label{eq:cond:nijyu}
\dif_t^2a(0,{\bar\xi})\neq 0.
\end{equation}
Indeed, assuming $a(0,{\bar\xi})=\dif_ta(0,{\bar\xi})=0$ it is easy to see
\[
{\rm det}\,\big(\lambda-F_p(0,0,{\bar\xi})\big)=\lambda^{2n}\big(\lambda^2-2\,\dif_t^2a(0,{\bar\xi})\big)
\]
which shows that $\dif_t^2a(0,{\bar\xi})\neq 0$ if $(0,0,{\bar\xi})$ is effectively hyperbolic.
From the Malgrange preparation theorem one can write, in some neighborhood ${\mathcal U}$ of $(0,{\bar\xi})$
\[
a(t,\xi)=e_2(t,\xi)\big(t^2+a_1(\xi)t+a_2(\xi)\big)=e_2\prod_{k=1}^2\big(t-{\nu}_k(\xi)\big)
\]
where $e_2>0$ and $a_i({\bar\xi})=0$. Note that if ${\mathsf{Re}}\,{\nu}_1(\xi)\neq {\mathsf{Re}}\,{\nu}_2(\xi)$ then ${\nu}_i(\xi)$ is necessarily real and ${\nu}_i(\xi)\leq 0$. In the case that  either ${\mathsf{Re}}\,{\nu}_1(\xi)\neq {\mathsf{Re}}\,{\nu}_2(\xi)$ or ${\mathsf{Re}}\,{\nu}_1(\xi)={\mathsf{Re}}\,{\nu}_2(\xi)\leq 0$ we take $\varphi_1=t$, $\Omega_1=[0,\delta]$ and $\varphi_2$ is absent. 
In the case  ${\mathsf{Re}}\,{\nu}_1(\xi)={\mathsf{Re}}\,{\nu}_2(\xi)=\psi(\xi)> 0$ so that 
\[
a(t,\xi)=e_2\big\{(t-\psi(\xi))^2+({\mathsf{Im}}\,\nu_1(\xi))^2\big\}
\]
 we take $
\varphi_1=\psi(\xi)-t$, $\Omega_1(\xi)=[0,\psi(\xi)]$, $\varphi_2=t-\psi(\xi)$, $\Omega_2(\xi)=[\psi(\xi),\delta]$. 
Repeating similar arguments as in \cite{N1}, \cite{N2}  one obtains
\begin{prop}
\label{pro:tdep:2}Assume that $p$ has a double characteristic root ${\bar \tau}$ at $(0,{\bar \xi})$, $|\bar{\xi}|=1$ and $(0,{\bar\tau}, {\bar \xi})$ is effectively hyperbolic if it is a critical point. Then one can find $\delta>0$ and a conic neighborhood $U$ of ${\bar\xi}$ such that for any $a_{j,\alpha}(t)$ with $j+|\alpha|\leq 1$ one can find $N_0\in \N$ such that for any $q\in \N$ with $q\geq N_0$ there is $C>0$ such that
\begin{equation}
\label{eq:matome:tdep:double}
\sum_{k=0}^{q+2}\big|\dif_t^k{\hat u}(t,\xi)\big|^2\leq C\sum_{k=0}^q\int_0^t\big|\lr{\xi}^{N_0+q-k}\dif_t^k{\hat f}(s,\xi)\big|^2ds
\end{equation}
for $(t,\xi)\in [0,\delta]\times U$ and for any ${\hat u}(t,\xi)$ with $\dif_t^k{\hat u}(0,\xi)=0$, $k=0,1$ and  ${\hat f}(t,\xi)$ with $\dif_t^k{\hat f}(0,\xi)=0$, $k=0,\ldots,q+N_0$ satisfying \eqref{eq:eqFT:2}.
\end{prop}
%

%%%%%%%%%%%%%%%%%%%%%%%%
\subsection{Proof of Theorem \ref{thm:tdep:kyosoku}}
\label{sec:proof:tdep:m}

We turn to the Cauchy problem  \eqref{eq:tdepCPm}. First note that, after Fourier transform in $x$, the equation is reduced to
\begin{equation}
\label{eq:tdepCPmFT}
\left\{\begin{array}{ll}
P(t,D_t,\xi){\hat u}=D_t^m{\hat u}+\sum_{j=0}^{m-1}\sum_{|\alpha|\leq m-j}a_{j,\alpha}(t)\,\xi^{\alpha}D_t^j{\hat u}=0,\\[6pt]
D_t^k{\hat u}(0,\xi)={\hat u}_k(\xi),\;\;\xi\in\R^n,\quad k=0,\ldots,m-1.
\end{array}\right.
\end{equation}
\begin{prop}
\label{pro:tdep:m}Assume that   every critical point $(0,\tau,\xi)$, $\xi\neq 0$ is effectively hyperbolic. Then there exists $\delta>0$ such that for any $a_{j,\alpha}(t)$ with $j+|\alpha|\leq m-1$  one can find $N_0, N_1\in \N$ and $C>0$ such that
\[
\sum_{k=0}^{m-1}\big|\lr{\xi}^{m-1-k}\dif_t^k{\hat u}(t,\xi)\big|^2\leq C\sum_{k=0}^{N_0}\int_0^t\big|\lr{\xi}^{N_0-k}\dif_t^k{\hat f}(s,\xi)\big|^2ds
\]
for $(t,\xi)\in [0,\delta]\times \R^n$ and for any ${\hat u}(t,\xi)$ with $\dif_t^k{\hat u}(0,\xi)=0$, $k=0,\ldots,m-1$ and ${\hat f}(t,\xi)$ with $\dif_t^k{\hat f}(0,\xi)=0$, $k=0,\ldots,N_1$ satisfying $P(t,D_t,\xi){\hat u}={\hat f}$.
\end{prop}
\begin{proof} Let ${\bar\xi}\neq 0$ be arbitrarily fixed. Write $
p(0,\tau,{\bar\xi})=\prod_{j=1}^r\big(\tau-\tau_j)^{m_j}$ 
where $\sum m_j=m$ and $\tau_j$ are real and  different each other, where $m_j\leq 3$ which follows from the assumption. There exist $\delta>0$ and a conic neighborhood $U$ of ${\bar\xi}$ such that one can write
\begin{align*}
&p(t,\tau,\xi)=\prod_{j=1}^rp^{(j)}(t,\tau,\xi),\\
&p^{(j)}(t,\tau,\xi)= \tau^{m_j}+a_{j,1}(t,\xi)\tau^{m_j-1}+\cdots+ a_{j,m_j}(t,\xi)
\end{align*}
for $(t,\xi)\in (-\delta,\delta)\times U$ where $a_{j,k}(t,\xi)$ are real valued, homogeneous of degree $k$ in $\xi$ and $p^{(j)}(0,\tau,{\bar\xi})=(\tau-\tau_j)^{m_j}$.  If $(0,\tau_j,{\bar\xi})$ is a critical point of $p$, and necessarily $m_j\geq 2$, then $(0,\tau_j,{\bar\xi})$ is a critical point of $p^{(j)}$ and it is easy to see
\[
H_p(0,\tau_j,{\bar\xi})=c_j\,H_{p^{(j)}}(0,\tau_j,{\bar\xi})
\]
with some $c_j\neq 0$ and hence $H_{p^{(j)}}(0,\tau_j,{\bar\xi})$ has non-zero real eigenvalues if $H_{p}(0,\tau_j,{\bar\xi})$ does and vice versa. It is well known  that one can write, in some conic neighborhood $U$ of ${\bar\xi}$ that
\[
P=P^{(1)}P^{(2)}\cdots P^{(r)}+R
\]
where $P^{(j)}$ are differential operators in $t$ of order $m_j$ with coefficients which are poly-homogeneous symbol in $\xi$ and $R$ is a differential operators in $t$ of order at most $m-1$ with $S^{-\infty}$ (in $\xi$) coefficients. Note that the principal symbol of $P^{(j)}$ is $p^{(j)}$ and hence the assumptions in Propositions \ref{pro:tdep} and \ref{pro:tdep:2} are satisfied. Therefore thanks to Propositions \ref{pro:tdep} and \ref{pro:tdep:2}  we have 
\begin{align*}
\sum_{k=0}^{q+m_j}\big|\lr{\xi}^{q+m_j-k}\dif_t^k{\hat u}(t)\big|^2\leq C\sum_{k=0}^q\Big\{\big|\lr{\xi}^{q-k}\dif_t^q(P^{(j)}{\hat u})(t)\big|^2\\
+\int_0^t\big|\lr{\xi}^{N+q-k}\dif_t^k(P^{(j)}{\hat u})(s)\big|^2ds\Big\}
\end{align*}
 in some conic neighborhood of ${\bar\xi}$ and for $j=1,\ldots,r$. Then by induction on $j=1,\ldots, r$ one obtains
\begin{align*}
\sum_{k=0}^{q+m}\big|\lr{\xi}^{q+m-k}\dif_t^k{\hat u}(t)\big|^2\leq C\sum_{k=0}^q\Big\{\big|\lr{\xi}^{q-k}\dif_t^kh(t)\big|^2
+\int_0^t\big|\lr{\xi}^{rN+q-k}\dif_t^kh(s)\big|^2ds\Big\}
\end{align*}
where $h(t)={\hat f}(t)-R\,{\hat u}(t)$. Note that for any $k, l\in\N$ there is $C_{k,l}$ such that
\[
\big|\dif_t^k(R\,{\hat u})(t)\big|\leq C_{k,l}\lr{\xi}^{-l}\sum_{j=0}^{k+m-1}\big|\lr{\xi}^{k+m-1-j}\dif_t^j{\hat u}(t)\big|^2.
\]
Therefore one  concludes that
\begin{align*}
\sum_{k=0}^{q+m}\big|\lr{\xi}^{q+m-k}\dif_t^k{\hat u}(t)\big|\leq C_1\sum_{k=0}^{q+m}\int_0^t\big|\lr{\xi}^{q+m-k}\dif_t^k{\hat u}(s)\big|^2ds\\
+C_2\sum_{k=0}^{q+1}\int_0^t\big|\lr{\xi}^{rN+q+1-k}\dif_t^k{\hat f}(s)\big|^2ds.
\end{align*}
Then the assertion  follows from the Gronwall's lemma. Finally applying a compactness arguments one can complete the proof.
\end{proof}

\noindent
Proof of Theorem \ref{thm:tdep:kyosoku}: Let $u_j(x)\in C_0^{\infty}(\R^n)$ and hence ${\hat u}_j(\xi)\in {\mathcal S}(\R^n)$. From $P{\hat u}=0$ one can determine $\dif_t^k{\hat u}(0,\xi)$ successively  from ${\hat u}_j(\xi)$. Take $N\geq N_1+m$ and define
\[
{\hat u}_N(t,\xi)=\sum_{k=0}^N\frac{t^k}{k!}\dif_t^k{\hat u}(0,\xi)
\]
which is in $C^{\infty}(\R;{\mathcal S}(\R^n))$. With ${\hat f}=-P{\hat u}_N$ it is clear that $\dif_t^k{\hat f}(0,\xi)=0$ for $k=0,\ldots,N_1$. Apply Proposition \ref{pro:tdep:m} to the  following Cauchy problem
\[
P{\hat w}=-P{\hat u}_N={\hat f}(t,\xi),\quad \dif_t^k{\hat w}(0,\xi)=0,\;k=0,\ldots,m-1
\]
to obtain
\[
\sum_{k=0}^{m-1}\big|\lr{\xi}^{m-1-k}\dif_t^k{\hat w}(t,\xi)\big|^2\leq C\sum_{k=0}^{N_0}\int_0^t\big|\lr{\xi}^{N_0-k}\dif_t^k{\hat f}(s,\xi)\big|^2ds.
\]
Since it is clear that
\[
\sum_{k=0}^{N_0}\big|\lr{\xi}^{N_0-k}\dif_t^k{\hat f}(s,\xi)\big|^2\leq C_{N_0,N_1} \sum_{j=0}^{m-1}\big|\lr{\xi}^{N+N_0-j}{\hat u}_j(\xi)\big|^2
\]
for $0\leq s\leq \delta$ then 
noting that ${\hat u}={\hat w}+{\hat u}_N$ is a solution to the Cauchy problem \eqref{eq:tdepCPmFT} one obtains
\[
\sum_{k=0}^{m-1}\big|\lr{\xi}^{m-1-k}\dif_t^k{\hat u}(t,\xi)\big|^2\leq C'\sum_{j=0}^{m-1}\big|\lr{\xi}^{N+N_0-j}{\hat u}_j(\xi)\big|^2.
\]
Therefore, by a Paley-Wiener Theorem we prove the $C^{\infty}$ well-posedness of the Cauchy problem \eqref{eq:tdepCPm}.

%
%%%%%%%%%%%%%%%%%%%%%%%%
\section{Third order hyperbolic operators with effectively hyperbolic critical points with two independent variables}
\label{sec:app:xdep}

In this section we consider the Cauchy problem for third order operators  with two independent variables in a neighborhood of the origin for $t\geq 0$;
\begin{equation}
\label{eq:xdepCP}
\left\{\begin{array}{ll}
D_t^3u+\sum_{j=0}^{2}\sum_{j+k\leq 3}a_{j,k}(t,x)D_x^kD_t^ju=0,\\[8pt]
D_t^ju(0,x)=u_j(x),\quad j=0,1,2
\end{array}\right.
\end{equation}
where the coefficients $a_{j,k}(t,x)$ ($j+k=3$) are assumed to be real valued real analytic in $(t,x)$ in a neighborhood of the origin and the principal symbol $p$
\[
p(t,x,\tau,\xi)=\tau^3+\sum_{j=0}^{2}\sum_{j+k=3}a_{j,k}(t,x)\xi^k\tau^j
\]
has only real roots in $\tau$ for  $(t,x)\in [0,T)\times U$ with some $T>0$ and a neighborhood $U$ of the origin.
\begin{theorem}
\label{thm:xdep:kyosoku}If every critical point $(0,0,\tau,1)$ is effectively hyperbolic then for any $a_{j,k}(t,x)$ with $j+k\leq 2$ which are $C^{\infty}$ near $(0,0)$  the Cauchy problem \eqref{eq:xdepCP} is $C^{\infty}$ well-posed near the origin for $t\geq 0$.
\end{theorem}
%

%%%%%%%%%%%%%%%%%%%%
\subsection{Key proposition, $x$ dependent case}

 Assume that $p$ has a triple characteristic root ${\bar\tau}$ at $(0,0,1)$ hence $(0,0,{\bar\tau},1)$ is a critical point. Making a suitable change of local coordinates $t=t'$, $x=x(t',x')$ such that $x(0,x')=x'$ one can assume that $a_{2,1}(t,x)=0$ so that
 \begin{equation}
 \label{eq:3form:x}
 p(t,x,\tau,\xi)=\tau^3-a(t,x)\xi^2\tau-b(t,x)\xi^3.
 \end{equation}
 Since the triple characteristic root is now ${\bar\tau}=0$ hence $b(0,0)=a(0,0)=0$ and the hyperbolicity condition implies that
 \begin{equation}
 \label{eq:cond:x}
 \Delta(t,x)=4\,a(t,x)^3-27\,b(t,x)^2\geq 0,\quad (t,x)\in [0,T)\times U.
 \end{equation}
 Note that $\dif_xa(0,0)=\dif_tb(0,0)=\dif_xb(0,0)=0$ which follows from Lemma \ref{lem:NiP} then it is clear that
 \[
{\rm det}\big(\lambda I- F_p(0,0,0,1)\big)
=\lambda^2\big(\lambda^2-(\dif_ta(0,0))^2\big).
\]
This implies $
\dif_ta(0,0)\neq 0$ 
since $(0,0,0,1)$ is assumed to be effectively hyperbolic.

A key proposition corresponding to Proposition \ref{pro:kihon} is obtained by applying similar arguments as in Section \ref{sec:tdep}   together with  some observations on non-negative real analytic functions with two independent variables given in \cite[Lemma 2.1]{N1} (see also \cite{N3}). We just give a sketch of the arguments. From the Weierstrass' preparation theorem there is a neighborhood  of $(0,0)$ where one can write
\[
\Delta(t,x)=e_2(t,x)\big\{t^3+a_1(x)t^2+a_2(x)t+a_3(x)\big\}
\]
where $e_2>0$ and $a_j(x)$ are real valued, real analytic with $a_j(0)=0$. Denote
\[
{\bar\Delta}(t,x)=t^3+a_1(x)t^2+a_2(x)t+a_3(x)
\]
then the next lemma corresponds to Lemma \ref{lem:Dkon}
\begin{lem}
\label{lem:Dform:x}There exists $\delta>0$ such that, in each interval $0<\pm \,x<\delta$,  one can write
\begin{equation}
\label{eq:Dform:x1}
{\bar\Delta}(t,x)=\big|t-\nu_2(x)\big|^2(t-\nu_1(x))
\end{equation}
where $\nu_1(x)$ is real valued with $\nu_1(x)\leq 0$ or
\begin{equation}
\label{eq:Dform:x2}
{\bar\Delta}(t,x)=\prod_{k=1}^3\big(t-\nu_k(x)\big)
\end{equation}
where $\nu_k(x)$ are real valued  with $\nu_k(x)\leq 0$, in both cases $\nu_j(x)$  are expressed as convergent Puiseux series;
\[
\nu_k(x)=\sum_{j\geq 0}C_{k,j}^{\pm}(\pm\,x)^{j/p_j},\quad ( \text{for some}\;\;p_j\in\N )
\]
on $0<\pm x<\delta$. In all cases there is $C>0$ such that 
\begin{equation}
\label{eq:difnu}
\big|d{\mathsf{Re}}\,\nu_j(x)/dx\big|\leq C,\quad 0<|x|<\delta.
\end{equation}
\end{lem}
Next we show a counterpart of Lemma \ref{lem:alnu}. Note that  one can write
\[
a(t,x)=e_1(t,x)\big(t+\alpha(x)\big)
\]
where $e_1>0$ and $\alpha(x)$ is real analytic with $\alpha(0)=0$ and $\alpha(x)\geq 0$ in $|x|<\delta$.
\begin{lem}
\label{lem:alnu:x}There exist $\varep>0$ and $\delta>0$ such that one can find $j^{\pm}\in \{1,2,3\}$ such that
\[
\varep^{-1}|\nu_{j^{\pm}}(x)|\geq \alpha(x),\quad 0<\pm x<\delta.
\]
\end{lem}
Recall that, choosing a smaller $\delta>0$ if necessary, one can assume that any two of ${\mathsf{Re}}\,\nu_j(x)$, ${\mathsf{Im}}\,\nu_j(x)$, $\nu(x)\equiv 0$ are either different or coincide in each interval $0<\pm x<\delta$.
Denote
\[
\psi(x)=\max{\big\{0,{\mathsf{Re}}\,\nu_2(x)\big\}},\quad |x|<\delta
\]
and define
\[
\begin{cases}
\phi_1=t,&\Omega_1=\{(t,x)\mid |x|<\delta, \; 0\leq t\leq \psi(x)/2\}\\
\phi_2=t-\psi(x),&\Omega_2=\{(t,x)\mid |x|<\delta,  \;\psi(x)/2\leq t\leq T\}
\end{cases}
\]
with a small $T>0$. If $\psi=0$ then $\phi_1=\phi_2=t$ and $\Omega_2=\{(t,x)\mid |x|<\delta, \; 0\leq t\leq T\}$.
 Now we have a key proposition corresponding to Proposition \ref{pro:kihon};
\begin{prop}
\label{pro:kihon:x} There exist $\delta>0$, $T>0$ and $C>0$ such that
\begin{equation}
\label{eq:kihon:x}
\phi_j^2\,a\leq C\Delta,\quad |\phi_j|\,|\dif_t\Delta|\leq C\Delta,\quad |\phi_j|\leq C a
\end{equation}
for $(t,x)\in \Omega_j$ and $j=1,2$.
\end{prop}
%

%%%%%%%%%%%%%%%%%%%%%%
\subsection{Energy estimates, $x$ dependent case}

 With $U={^t}(D_t^2u,D_xD_tu, D_x^2u)$ the equation \eqref{eq:xdepCP} can be written
\begin{equation}
\label{eq:eqU:x}
\begin{split}
\frac{\dif}{\dif t}U&=\begin{bmatrix}0&a(t,x)&b(t,x)\\
                                    1&0&0\\
                                    0&1&0\\
                                    \end{bmatrix}\!\frac{\dif}{\dif x}U\\[3pt]
                                    &+i\begin{bmatrix}
                                    b_1(t,x)&b_2(t,x)&b_3(t,x)\\
                                    0&0&0\\
                                    0&0&0\\
                                    \end{bmatrix}U+\begin{bmatrix}i{ f}\;\\
                                    0\\
                                    0\\
                                    \end{bmatrix}
                                   \\[4pt]
                                   &=A(t,x)\dif_xU+BU+F.
                                    \end{split}
\end{equation}
Let $T(t,x)$ be the orthonormal matrix introduced in Section \ref{sec:defT} such that with $V=T^{-1}U$ the equation \eqref{eq:eqU:x} becomes
\begin{equation}
\label{eq:eqV:x}
\dif_tV={\mathcal A}\dif_xV+{\mathcal B}V+{\tilde F}
\end{equation}
where 
\[
{\mathcal A}=T^{-1}AT,\quad {\mathcal B}=(\dif_tT^{-1})T-{\mathcal A}(\dif_xT^{-1})T+T^{-1}BT,\quad {\tilde F}=T^{-1}F.
\]
Let $\omega$ be an open domain in $\R^2$ and let $g\in C^1(\omega)$ be a positive scalar function. Denote by $\dif\omega$ the boundary of $\omega$ equipped with the usual orientation.  Let
\[
G(V)=g\lr{\Lambda V,V}dx+g\lr{\Lambda{\mathcal A}V,V}dt
\]
then one has
\begin{equation}
\label{eq:ene:tosiki}
\begin{split}
2\,{\mathsf{Re}}\,\int_{\omega}g&\lr{\Lambda V,{\mathcal B}V+{\tilde F}}\,dxdt=-\int_{\dif\omega}G(V)\\
&-\int_{\omega}\Big\{(\dif_tg)\lr{\Lambda V,V}+g\lr{(\dif_t\Lambda)V,V}\Big\}\,dxdt\\
&+\int_{\omega}\Big\{(\dif_xg)\lr{\Lambda{\mathcal A}V,V}+g\lr{\dif_x(\Lambda{\mathcal A})V,V}\Big\}\,dxdt.
\end{split}
\end{equation}
Here make a remark on the boundary term. Denote
\[
\tau_{max}=\max_{(t,x)\in [0,T]\times U}\big|\tau_j(t,x)\big|
\]
where $\tau_j(t,x)$, $j=1,2,3$ are characteristic roots of $p(t,x,\tau,1)$. 
\begin{lem}
\label{lem:spacelike} Let $\Gamma: [a,b]\ni x\mapsto (f(x),x)$ be a space-like curve, that is
\[
1>\tau_{max}\,\big|f'(x)\big|,\quad x\in [a,b].
\]
Then one has
\[
\int_{\Gamma}G(V)\geq 0.
\]
\end{lem}
\begin{proof}
Since $g(t,x)>0$ is scalar function it suffices to prove
\[
\lr{\Lambda(f(x),x)V,V}+\lr{f'(x)\Lambda(f(x),x){\mathcal A}(f(x),x)V,V}\geq 0,\;\;\forall V\in \C^3.
\]
To simplify notation we denote $\Lambda(f(x),x)$ and ${\mathcal A}(f(x),x)$ by  just $\Lambda$ and ${\mathcal A}$. Noting that $\Lambda{\mathcal A}={^t}\!{\mathcal A}\Lambda$ one has
\[
\big|\lr{f'\Lambda{\mathcal A}V,V}\big|=\big|f'\lr{\Lambda V,{\mathcal A}V}\big|\leq \lr{\Lambda V,V}^{1/2}\lr{\Lambda{\mathcal A}V,{\mathcal A}V}^{1/2}|f'|.
\]
Therefore to prove the assertion it suffices to show 
\[
\lr{\Lambda{\mathcal A}V,{\mathcal A}V}|f'|^2\leq \lr{\Lambda V,V}
\]
that is, the maximal eigenvalue of ${|f'|^2}({^t\!}{\mathcal A}\Lambda{\mathcal A})$ with respect to $\Lambda$ is at most $1$. From $^t\!{\mathcal A}\Lambda=\Lambda{\mathcal A}$ it follows that
\begin{align*}
{\rm det}\big(\lambda \Lambda-|f'|^2(^t\!{\mathcal A}\Lambda{\mathcal A})\big)={\rm det}\big(\lambda \Lambda-|f'|^2\Lambda{\mathcal A}^2)\big)\\
=({\rm det}\,\Lambda)\,{\rm det}\big(\lambda I-|f'|^2{\mathcal A}^2\big).
\end{align*}
Since $\tau_j$ are the eigenvalues of ${\mathcal A}$ we see that the eigenvalues  of $|f'|^2{\mathcal A}^2$ is at most $|f'(x)|^2\,\tau_{max}^2<1$ hence the assertion.
\end{proof}
Define $\Omega=\{(t,x)\mid |x|\leq \delta(T-t), 0\leq t\leq T\}$ and 
\begin{equation}
\label{eq:space:like}
\begin{split}
&\varphi_1=t,\quad \omega_1=\{(t,x)\mid |x|\leq \delta (T-t), 0\leq t\leq \psi(x)/2\},\\
&\varphi_2=\psi(x)-t,\quad \omega_2=\{(t,x)\mid |x|\leq \delta(T-t), \psi(x)/2\leq t\leq \psi(x)\},\\
&\varphi_3=t-\psi(x),\quad \omega_3=\{(t,x)\mid |x|\leq \delta(T-t), \psi(x)\leq t\leq T\}
\end{split}
\end{equation}
where $\delta>0, T>0$ are small such that the lines $|x|=\delta(T-t)$, $0\leq t\leq T$ are space-like. 
Thanks to Proposition \ref{pro:kihon:x}  and Lemma \ref{lem:Dform:x} it follows that
\begin{equation}
\label{eq:katei:x}
\varphi_j^2\leq C\lambda_1,\;\;\varphi_j|\dif_t\lambda_1|\leq C\lambda_1,\;\;\varphi_j\leq C\lambda_2,\;\;\big|\dif_x\varphi_j\big|\leq C,\quad (t,x)\in \omega_j.
 \end{equation}
Apply \eqref{eq:ene:tosiki} with 
\begin{align*}
g=g_j=\varphi_j^{2(-1)^jN-1},\;\;
G(V)= G_j(V)=g_j\lr{\Lambda V,V}dx+g_j\lr{\Lambda{\mathcal A}V,V}dt
\end{align*}
and $\omega=\omega_j$ then from the arguments in Section \ref{sec:enehyoka:1} one obtains
\begin{lem}
\label{lem:enehyoka:x}There exist $N_0$, $C>0$ such that 
\[
C\int_{\omega_j}\varphi_jg_j\|F\|^2dxdt\geq -\int_{\dif \omega_j}G_j(V)+N\int_{\omega_j}\varphi_j^{-1}g_j\lr{\Lambda V,V}dxdt
\]
for $N\geq N_0$.
\end{lem}
Repeating the same arguments as in \cite[Lemma 3.1]{N1} (also \cite{N3}) one has
\begin{prop}
\label{pro:enehyoka:x}There exists $C>0$ such that for every $n\in\N$ one can find $N_1$ such that
\begin{align*}
\sum_{k+\ell\leq n}\int_{\omega_j}g_j\varphi_j\|\dif_t^k\dif_x^{\ell}U\|^2dxdt-\sum_{\ell\leq n}\int_{\dif\omega_j}G_j(T^{-1}\dif_x^{\ell}U)\\
\leq C\sum_{k+\ell\leq n}\int_{\omega_j}g_j\varphi_j\|\dif_t^k\dif_x^{\ell}LU\|^2dxdt
\end{align*}
for any $N\geq N_1$.
\end{prop}
Following the  same arguments as in \cite{N1}, \cite{N3}
one can collect energy estimates in each $\omega_j$ to obtain 
\begin{prop}
\label{pro:xdep:3} Assume that $p$ has a triple characteristic root ${\bar\tau}$ at $(0,0,1)$ and $(0,0,{\bar\tau},1)$ is effectively hyperbolic. Then there exist $T>0, \delta>0$ such that for any $a_{j,k}(t,x)\in C^{\infty}(\Omega)$, $j+k\leq 2$  one can find $C>0$ and $Q\in\N$ such that
\begin{equation}
\label{eq:Omega:x}
\sum_{k+\ell\leq n}\int_{\Omega}\|\dif_t^k\dif_x^{\ell}U\|^2dxdt 
\leq C\sum_{k+\ell\leq Q}\int_{\Omega}\|\dif_t^k\dif_x^{\ell}LU\|^2dxdt.
\end{equation}
for any $U(t,x)\in C^{\infty}(\Omega)$ with $\dif_t^kU(0,x)=0$, $k=0,\ldots, Q$.
\end{prop}
%
%
%%%%%%%%%%%%%%%%%%%%%%%
\subsection{Proof of Theorem \ref{thm:xdep:kyosoku}}

To complete the proof of Theorem \ref{thm:xdep:kyosoku}, study the remaining  case that $p$ has a double characteristic root at $(0,0,1)$.
\begin{prop}
\label{pro:xdep:2} Assume that $p$ has a double characteristic root ${\bar\tau}$ at $(0,0,1)$ such that $(0,0,{\bar\tau},1)$ is effectively hyperbolic  if it is a critical point. Then the same assertion as in Proposition \ref{pro:xdep:3} holds.
\end{prop}
We give a sketch of the proof. Assume that $p$ has a double characteristic root ${\bar\tau}$ at $(0,0,1)$ and hence, after  a suitable change of local coordinates, one can write
\begin{equation}
\label{eq:xdep:2:form}
p(t,x,\tau,\xi)=\big(\tau-b(t,x)\xi\big)\big(\tau^2-a(t,x)\xi^2\big)=p_1p_2
\end{equation}
where $p_1=\tau-b(t,x)\xi$, $p_2=\tau^2-a(t,x)\xi^2$ and $a(0,0)=0$, $b(0,0)\neq0$. 

If $(0,0,0,1)$ is a critical point of $p$ and hence $\dif_ta(0,0)=0$ then $H_p=c\,H_{p_2}$ at $(0,0,0,1)$  with some $c\neq 0$ and 
${\rm det}\big(\lambda I-F_{p_2}(0,0,0,1)\big)=\lambda^2\big(\lambda^2-2\,\dif_t^2a(0,0)\big)$ which shows $
\dif_t^2a(0,0)\neq 0$. From the Weierstrass' preparation theorem one can write
\[
a(t,x)=e_2(t,x)\big(t^2+2{\tilde a}_1(x)t+{\tilde a}_2(x)\big)=e_2(t,x)\Delta_2(t,x)
\]
where $e_2>0$, ${\tilde a}_j(0)=0$ and $\Delta_2$ takes the form, in each $0<\pm x<\delta$, either  \eqref{eq:Dform:x1} or \eqref{eq:Dform:x2}  where  $\nu_1$ is absent. If $\Delta_2$ has the form \eqref{eq:Dform:x2} or \eqref{eq:Dform:x1} with ${\mathsf{Re}}\,\nu_2(x)\leq 0$ it suffices to take $\varphi_1=t$, $\omega_1=\Omega$ ($\varphi_2$ is not needed). If $\Delta_2$ takes the form \eqref{eq:Dform:x1} with ${\mathsf{Re}}\,\nu_2(x)=\psi(x)> 0$ we choose $\varphi_1=\psi(x)-t$, $\omega_1=\{(t,x)\mid |x|\leq \delta(T-t), 0\leq t\leq \psi(x)\}$ and $\varphi_2=t-\psi(x)$, $\omega_2=\{(t,x)\mid |x|\leq \delta(T-t), \psi(x)\leq t\leq T\}$. 

If $\dif_ta(0,0)\neq0$ one can write
\[
a(t,x)=e_1(t,x)\big(t+\alpha(x)\big),\quad \alpha(0)=0
\]
where $e_1>0$ and $\alpha(x)\geq 0$ in a neighborhood of $x=0$. In this case we take $\varphi_1=t$, $\omega_1=\Omega$ ($\varphi_2$ is not needed). 

Denote
\[
P_2=D_t^2-a(t,x)D_x^2,\quad P_1=D_t-b(t,x)D_x.
\]
Then repeating the same arguments as in \cite{N1}, \cite{N3} one has
\begin{equation}
\label{eq:ene:P:2}
\int_{\omega_j}\varphi_jg_j|P_2u|^2dxdt\geq -\int_{\dif \omega_j}G^{(2)}_j(u)
+N\int_{\omega_j}\varphi_jg_j\big(|D_tu|^2+|D_xu|^2\big)dxdt
\end{equation}
for $N\geq N_0$ with $G^{(2)}_j(u)=g_j\big(|\dif_tu|^2+a|\dif_x^2u|^2\big)dx+ag_j\big(\dif_xu\cdot \overline{\dif_tu}+\overline{\dif_xu}\cdot\dif_tu\big)dt$. On the other hand, since $P_1$ is a first order differential operator with a real valued $b(t,x)$ it is easy to see that
\begin{equation}
\label{eq:ene:P:1}
\begin{split}
\int_{\omega_j}\varphi_jg_j|P_1u|^2dxdt\geq -\int_{\dif \omega_j}G^{(1)}_j(u)
+N\int_{\omega_j}\varphi_j^{-1}g_j|u|^2dxdt\\
\geq -\int_{\dif \omega_j}G^{(1)}_j(u)
+N\int_{\omega_j}\varphi_jg_j|u|^2dxdt
\end{split}
\end{equation}
for $N\geq N_0$ with $G^{(1)}_j(u)=g_j|u|^2dx+bg_j|u|^2dt$. Inserting $u=P_1u$ in \eqref{eq:ene:P:2} and $u=P_2u$ in \eqref{eq:ene:P:1} respectively and adding them one obtains
\begin{equation}
\label{eq:P1P2}
\begin{split}
\int_{\omega_j}\varphi_jg_j\big(|P_2P_1u|^2+|P_1P_2u|^2\big)dxdt\geq -\int_{\dif \omega_j}{\tilde G}_j(u)\\
+N\int_{\omega_j}\varphi_jg_j\big(|D_tP_1u|^2+|D_xP_1u|^2+|P_2u|^2\big)dxdt.
\end{split}
\end{equation}
In view of $b(0,0)\neq 0$ and $a(0,0)=0$ it is easy to see that $D_t^2$, $D_xD_t$ and $D_x^2$ are linear combinations of $D_tP_1$, $D_xP_1$ and $P_2$ with smooth coefficients modulo first order operators. Since one can write
\[
P=P_2P_1+\sum_{i+j\leq 2}b_{i,j}(t,x)D_x^iD_t^j=P_1P_2+\sum_{i+j\leq 2}{\tilde b}_{i,j}(t,x)D_x^iD_t^j
\]
one concludes from \eqref{eq:P1P2} that
\[
\int_{\omega_j}\varphi_jg_j|Pu|^2dxdt\geq -\int_{\dif \omega_j}{\tilde G}_j(u)
+N\sum_{i+j\leq 2}\int_{\omega_j}\varphi_jg_j|D_t^iD_x^ju|^2dxdt.
\]
The rest of the proof is parallel to that of Proposition \ref{pro:xdep:3}.

\medskip
\noindent
Proof of Theorem \ref{thm:xdep:kyosoku}: Note that \eqref{eq:Omega:x} implies
\[
\sum_{k+\ell\leq n+2}\int_{\Omega}\|\dif_t^k\dif_x^{\ell}u\|^2dxdt 
\leq C\sum_{k+\ell\leq Q}\int_{\Omega}\|\dif_t^k\dif_x^{\ell}Pu\|^2dxdt.
\]
Then applying an approximation argument with the Cauchy-Kowalevsky theorem one can conclude the proof of Theorem \ref{thm:xdep:kyosoku}. 

\medskip
We restrict ourselves to third order operators in Theorem \ref{thm:xdep:kyosoku} because it seems to be hard to apply the same arguments as in Section \ref{sec:proof:tdep:m} to this case.

%%%%%%%%%
%%%%%%%%%%%%%
\section{Example of third order homogeneous equation with general triple characteristics}

To show that the  same arguments in the previous sections can be applicable to hyperbolic operators with more general triple characteristics, we study the Cauchy problem 
\begin{equation}
\label{eq:xdepCP:unif}
\left\{\begin{array}{ll}
D_t^3u-a(t,x)D_tD_x^2u-b(t,x)D_x^3u=0,\;\;U\cap\{t> s\}\\[8pt]
D_t^ju(s,x)=u_j(x),\quad j=0,1,2
\end{array}\right.
\end{equation}
in a {\it full neighborhood} of $(0,0)$ in $\R^2$.
The Cauchy problem \eqref{eq:xdepCP:unif} is (uniformly) well posed near the origin if one can find a neighborhood $U$ of $(0,0)$ and a small $\ep>0$ such that for any $|s|<\ep$ and  any $u_j(x)\in C^{\infty}(\Omega\cap\{t=s\})$ there is a unique solution to \eqref{eq:xdepCP:unif}. Assume that there is a neighborhood $U'$ of $(0,0)$ such that $p(t,x,\tau,1)=0$ has only real roots in $\tau$ for  $(t,x)\in U'$, that is
\begin{equation}
\label{eq:ipan:triple}
\Delta(t,x)=4\,a(t,x)^3-27\,b(t,x)^2\geq 0, \quad (t,x)\in U'
\end{equation}
 which is necessary for the Cauchy problem \eqref{eq:xdepCP:unif} to be well posed near the origin (\cite{Lax, Mbook}).  
\begin{theorem}
\label{thm:txdep:CP}Assume \eqref{eq:ipan:triple} and that there exists $C>0$ such that 
\begin{equation}
\label{eq:exam:cond}
a^3\leq C\Delta, \qquad |\dif_tb|\leq C\sqrt{a}\,|\dif_ta|
\end{equation}
holds in a neighborhood of $(0,0)$. Then  the Cauchy problem \eqref{eq:xdepCP:unif} is $C^{\infty}$ well posed near the origin. 
\end{theorem}
Note that if both $a(t,x)$ and $b(t,x)$ are independent of $t$, Theorem \ref{thm:txdep:CP} is a very special case of 
\cite[Theorem 1.1]{SpaTa} and if $a(t,x)$ and $b(t,x)$ are independent of $x$ this is also a special case of \cite[Theorem 2]{CoOr}.

\medskip
We give a rough sketch of the proof.  If $\Delta(0,0)>0$ then $p$ is strictly hyperbolic near the origin and the assertion is clear so $\Delta(0,0)=0$ is assumed from now on and hence  $a(0,0)=0$ by assumption \eqref{eq:exam:cond}. Then $p$ has a triple characteristic root at $(0,0,1)$.  Thanks to Proposition \ref{pro:Skon} the eigenvalues $\lambda_j$ of the B\'ezoutian matrix satisfy
\[
\lambda_1\simeq a^2,\quad \lambda_2\simeq a,\quad \lambda_3\simeq 1
\]
for $\Delta/a\geq a^2/C$. Let $\phi$ be a scalar function satisfying $\phi>0$ and $\dif_t\phi>0$ in $\omega$ and consider the following energy;
\begin{equation}
\label{eq:ene:form:2}
(\phi^{-N}e^{-\gamma t} \Lambda V,V)=\int_{\omega} \phi^{-N}e^{-\gamma t}\lr{\Lambda V,V}dx
\end{equation}
where $N>0, \gamma>0$ are positive parameters.  Note that 
\begin{align*}
\frac{d}{dt}\,(e^{-\gamma t}\phi^{-N}\, \Lambda V,V)=-\big((N\phi^{-1}\dif_t\phi+\gamma)e^{-\gamma t}\phi^{-N}\,\Lambda V,V\big)
\\+\big(e^{-\gamma t}\phi^{-N}(\dif_t\Lambda)V,V\big)
+2\,{\mathsf{Re}}\,\big(e^{-\gamma t}\phi^{-N}\Lambda({\mathcal A}\dif_xV+{\mathcal B}V),V\big)
\end{align*}
and
\begin{equation}
\label{eq:ene:bis}
(N\phi^{-1}\dif_t\phi+\gamma)\lr{\Lambda V,V}=N\sum_{j=1}^3\phi^{-1}\dif_t\phi\lambda_j |V_j|^2+\gamma\sum_{j=1}^3\lambda_j |V_j|^2.
\end{equation}
Since the hyperbolicity condition $4a^3\geq 27b^2$ is assumed  in a full neighborhood of $(0,0)$ then Lemma \ref{lem:NiP} now states
\begin{equation}
\label{eq:dif:ab:imp}
|\partial_ta|\preceq \sqrt{a},\quad|\partial_xa|\preceq \sqrt{a},\quad |\partial_xb|\preceq a,\quad |\partial_tb|\preceq a.
\end{equation}
Suppose that a scalar weight $\phi$ satisfying the following property is obtained;
\begin{equation}
\label{eq:katei:Dbis}
\phi\,|\dif_ta|\preceq a\,\dif_t\phi,\qquad \phi\preceq \dif_t\phi\quad \text{in}\quad \omega.
\end{equation}
Since $|\dif_tq(\lambda_j)|\preceq (|\dif_ta|+|b||\dif_tb|)\lambda_j+|\dif_t\Delta|$ and $|\dif_t\Delta|\preceq a^2\,|\dif_ta|$ by \eqref{eq:exam:cond} then
\begin{equation}
\label{eq:dif:t:q}
|\dif_t\lambda_j|\preceq \frac{(|\dif_ta|+a^{5/2})\lambda_j+a^2\,|\dif_ta|}{|q_{\lambda}(\lambda_j)|}
\end{equation}
and hence $
|\dif_t\lambda_j|\preceq (|\dif_ta|/a)\lambda_j+a^{3/2}\,\lambda_j+a\,|\dif_ta|$ for $j=1,2$ 
which shows that
\[
\phi\,|\dif_t\lambda_j|\preceq \dif_t\phi\,\big(\lambda_j+a^2\big).
\]
Thus $\phi\,|\dif_t\lambda_j|\preceq \dif_t\phi\,\lambda_j$ for $j=1,2,3$ since $\lambda_2\simeq a$. The case $j=3$ is clear from \eqref{eq:katei:Dbis}. This proves that $\big|\phi^{-N}\lr{(\dif_t\Lambda)V,V}\big|$ is bounded by \eqref{eq:ene:bis} taking $N$ large. Next assume that a scalar weight $\phi$ satisfies
\begin{equation}
\label{eq:Ome:to:0}
\sup_{(t,x)\in\omega_{\ep}}\sqrt{a}\,|\dif_x\phi|(\dif_t\phi)^{-1}=0\quad (\ep\to 0)
\end{equation}
where $\omega_{\ep}$ is a family of regions converging to $(0,0)$ as $\ep\to 0$. Recall \eqref{eq:DAdifV};
\[
2\,{\mathsf{Re}}\,(\phi^{-N}\Lambda{\mathcal A}\dif_x V,V)=N\big(\phi^{-N-1}(\dif_x\phi)\Lambda{\mathcal A}V,V\big)-\big(\phi^{-N}\dif_x(\Lambda{\mathcal A})V,V).
\]
 Thanks to \eqref{eq:adphi} the term $N\big|\phi^{-N-1}(\dif_x\phi)\lr{\Lambda{\mathcal A}V,V}\big|$ is bounded by \eqref{eq:ene:bis} in $\omega_{\ep}$ for enough small $\ep$ in virtue of  the assumption \eqref{eq:Ome:to:0}. Consider the term $\lr{\dif_x(\Lambda{\mathcal A})V,V}=\lr{(\dif_x\Lambda){\mathcal A}V,V}+\lr{(\dif_x{\mathcal A})V,\Lambda v}$. From \eqref{eq:dla1:la1} one has $
|\dif_x\lambda_1|\preceq a^{3/2}$. Since $|\dif_x\lambda_2|=O(\sqrt{a})$ and $|\dif_x\lambda_3|=O(1)$ it follows from Lemmas \ref{lem:difTAT}
\[
(\dif_x\Lambda){\mathcal A}=\begin{bmatrix}
O(a^2)&O(a^{3/2})&O(a^2)\\
O(a^{3/2})&O(a)&O(\sqrt{a})\\
O(a^{3/2})&O( a)&O( a^{5/2})
\end{bmatrix}
\]
and then $|\lr{(\dif_x\Lambda){\mathcal A}V,V}|$ is bounded by
\begin{align*}
a^2|V_1|^2+a|V_2|^2+a^{5/2}|V_3|^2+a^{3/2}|V_1||V_2|+a^{3/2}|V_1||V_3|+\sqrt{a}|V_2||V_3|\\
\preceq a^2|V_1|^2+a|V_2|^2+|V_3|^3.
\end{align*}
Consider $\big|\lr{(\dif_x{\mathcal A})V,\Lambda V}\big|$. Thanks to Lemma \ref{lem:difTAT}  one has 
\[
\Lambda(\dif_x{\mathcal A})=\begin{bmatrix}
O(a^2)&O(a^{3/2})&O(a^2)\\
O(a^{3/2})&O(a)&O(\sqrt{a})\\
O( a)&O( a)&O(\sqrt{a})
\end{bmatrix}.
\]
Thus $|\lr{\Lambda(\dif_x{\mathcal A})V,V}|$ is bounded by
\begin{align*}
a^2|V_1|^2+a|V_2|^2+\sqrt{a}|V_3|^2+a^{3/2}|V_1||V_2|+a|V_1||V_3|+\sqrt{a}|V_2||V_3|\\
\preceq a^2|V_1|^2+a|V_2|^2+|V_3|^3.
\end{align*}
Then $|(e^{-\gamma t}\phi^{-N}\Lambda{\mathcal A}\dif_x V,V)|$ is bounded by \eqref{eq:ene:bis} taking $\ep>0$ small and $\gamma$ large.

Turn to ${\mathsf{Re}}\,\big(e^{-\gamma t}\phi^{-N}\Lambda{\mathcal B}V,V\big)$ where ${\mathcal B}=(\dif_tT^{-1})T-{\mathcal A}(\dif_xT^{-1})T$. Using \eqref{eq:dif:t:q} it is easy to see
\begin{equation}
\label{eq:hyoka:el:tbis}
\begin{split}
|\dif_t\ell_{11}|\preceq a|\dif_ta|+a^2,\;\; |\dif_t\ell_{21}|\preceq |\dif_tb|+a^{5/2},\;\; |\dif_t\ell_{31}|\preceq |\dif_ta|+a,\\
|\dif_t\ell_{12}|\preceq a^{2},\;\; |\dif_t\ell_{22}|\preceq |\dif_ta|+a^{3/2},\;\; |\dif_t\ell_{32}|\preceq |\dif_tb|+a^2,\\
|\dif_t\ell_{13}|\preceq |\dif_ta|+a^{5/2},\;\; |\dif_t\ell_{23}|\preceq a^{2},\;\; |\dif_t\ell_{33}|\preceq |\dif_ta|.
\end{split}
\end{equation}
Noting $|\dif_tb|\preceq \sqrt{a}\,|\dif_ta|$ these estimates improve \eqref{eq:dT:-1:T} to
\begin{lem} Let $(\dif_tT^{-1})T=({\tilde t}_{ij})$ then
\[
({\tilde t}_{ij})=\begin{bmatrix}0&O(a+|\dif_ta|/\sqrt{a})&O(|\dif_ta|+a^{5/2}))\\
O(a+|\dif_ta|/\sqrt{a})&0&O(\sqrt{a}|\dif_ta|+a^2)\\
O(|\dif_ta|+a^{5/2})&O(\sqrt{a}|\dif_ta|+a^2)&0\\
\end{bmatrix}.
\]
\end{lem}
In order to estimate $\big|\phi^{-N}\lr{(\dif_tT^{-1})TV,\Lambda V}\big|$ recalling $\Lambda\simeq {\rm diag}(a^2,a,1)$ it suffices to estimate
\[
|\phi^{-N}|\Big(a\,{\tilde t}_{21}|V_1|| V_2|+{\tilde t}_{31}|V_1||V_3|+{\tilde t}_{32}|V_2||V_3|\Big).
\]
Note that $a\,{\tilde t}_{21}\preceq a^2+\sqrt{a}\,|\dif_ta|$ and
\begin{gather*}
a^2|V_1||V_2|\preceq \sqrt{a}\,(a^2|V_1|^2+a|V_2|^2),\\
\sqrt{a}\,|\dif_ta||V_1||V_2|\preceq 
\phi^{-1}\dif_t\phi\,a^{3/2} |V_1||V_2|
\preceq \phi^{-1}\dif_t\phi\big(a^2 |V_1|^2+a |V_2|^2\big).
\end{gather*}
As for ${\tilde t}_{31}|V_1||V_3|\preceq (|\dif_ta|+a^{5/2})|V_1||V_3|$  one has
\begin{gather*}
a^{5/2}\,|V_1||V_3|\preceq a^{3/2}\big(a^2|V_1|^2+|V_3|^2\big),\\
|\dif_ta|\,|V_1||V_3|\preceq \phi^{-1}\dif_t\phi\,a|V_1||V_3|
\preceq \phi^{-1}\dif_t\phi\big(a^2|V_1|^2+|V_3|^2\big).
\end{gather*}
Noting ${\tilde t}_{32}|V_2||V_3|\preceq (\sqrt{a}|\dif_ta|+a^2)|V_2||V_3|\preceq a\,|V_2||V_3|\preceq \sqrt{a}(a|V_2|^2+|V_3|^2)$
one concludes that $\big|\phi^{-N}\lr{(\dif_tT^{-1})TV,\Lambda V}\big|$ is bounded by \eqref{eq:ene:bis} taking $N$ large and $\gamma\geq 1$.  Consider $|\phi^{-N}\lr{(\dif_xT^{-1})TV,\Lambda{\mathcal A}V}$. From Lemma \ref{lem:difTAT}  it follows that
\[
\Lambda{\mathcal A}=\begin{bmatrix}
O(a^{5/2})&O(a^2)&O(a^{5/2})\\
O(a^2)&O(a^{3/2})&O(a)\\
O(a^{5/2})&O(a)&O(a^{5/2})\\
\end{bmatrix}
\]
because $\Lambda{\mathcal A}$ is symmetric and $\Lambda\simeq {\rm diag}(a^2,a,1)$. Taking \eqref{eq:dx:TtT} into account this shows that
\[
\Lambda{\mathcal A}(\dif_xT^{-1})T=\begin{bmatrix}
O(a^2)&O(a^{5/2})&O(a^3)\\
O(a^{3/2})&O(a^2)&O(a^{5/2})\\
O(a)&O(a^{3/2})&O(a^2)\\
\end{bmatrix}.
\]
Thus $|\lr{(\dif_xT^{-1})TV,\Lambda{\mathcal A}V}|$ is bounded by
\begin{align*}
a^2\sum_{j=1}^3|V_j|^2+a^{3/2}|V_1||V_2|+a|V_1||V_3|+a^{3/2}|V_2||V_3|\\
\preceq a^2|V_1|^2+a|V_2|^2+|V_3|^2
\end{align*}
so that $|{\mathsf{Re}}\,\big(e^{-\gamma t}\phi^{-N}\Lambda{\mathcal B}V,V\big)|$ is bounded by \eqref{eq:ene:bis} taking $N$ and $\gamma$ large. 

Therefore to obtain energy estimates it suffices to find pairs $(\phi_j,\omega_{j,\ep})$ of scalar weight $\phi_j$ and subregion $\omega_{j,\ep}$  such that \eqref{eq:katei:Dbis} and \eqref{eq:Ome:to:0} are verified and $\cup\, \omega_{j,\ep}$ covers a neighborhood of $(0,0)$ for any fixed small $\ep>0$.  Since the choice of $(\phi_j,\omega_{j,\ep})$ is exactly same as in \cite{N1} (also \cite{N3}) we only mention how to choose $\phi_j$ and $\omega_{j,\ep}$ ($\phi$ is denoted by $\rho$ in \cite{N1} or \cite{N3}).

Following \cite{N1} one can define a real valued function $\alpha(t,x)$
\begin{equation}
\label{eq:al}
\alpha(t,x)=x^{n}\prod_{i\in I_1}(t-t_i(x))\prod_{i\in I_2}|t-t_i(x)|e(t,x)
\end{equation}
so that $a(t,x)=\alpha(t,x)^2$ where $t_i(x)$ has a  convergent Puiseux expansion in $0<\pm x<\delta$ with small $\delta$ and  ${\mathsf{Im}}\,t_i(x)\neq 0$ if $i\in I_2$. We choose all distinct $t_k(x)$ in \eqref{eq:al} and rename them as $t_1(x),\ldots,t_m(x)$. Taking $\delta$ small one can assume that
\begin{align*}
&{\mathsf{Re}}\,t_{\mu_1}(x)\leq {\mathsf{Re}}\,t_{\mu_2}(x)\leq \cdots \leq {\mathsf{Re}}\,t_{\mu_m}(x),\quad 0<x<\delta,\\
&{\mathsf{Re}}\,t_{\nu_1}(x)\leq {\mathsf{Re}}\,t_{\nu_2}(x)\leq \cdots \leq {\mathsf{Re}}\,t_{\nu_m}(x),\quad -\delta<x<0.
\end{align*}
Define $\sigma_j(x)$ by
\[
\sigma_j(x)={\mathsf{Re}}\,t_{\mu_j}(x)\;\; \text{for}\;\; x>0, \quad \sigma_j(x)={\mathsf{Re}}\,t_{\nu_j}(x)\;\; \text{for}\;\; x<0
\]
so that $\sigma_1(x)\leq \cdots\leq \sigma_m(x)$ in $|x|<\delta$. Define
\[
s_j(x)=\frac{\sigma_j(x)+\sigma_{j+1}(x)}{2},\;1\leq j\leq m-1,\;s_0(x)=-3t^*(x),\;s_m(x)=3t^*(x)
\]
with 
\[
t^*(x)=\big(\sum_{j=1}^m|t_j(x)|^2\big)^{1/2}
\]
where the sum is taken over all distinct $t_i(x)$ in \eqref{eq:al}.  Denote by $\omega_j^{\pm}$ and $\omega(T)$ the  subregions defined by
\begin{align*}
&\omega_j=\{(t,x)\mid  |x|\leq {\bar\delta}(T-t), s_{j-1}(x)\leq t\leq s_j(x)\}\;\; (j=1,\ldots,m),\\
&\omega^{\pm}_j=\omega_j\cap\{t\gtrless \sigma_j\},\;\;\omega(T)=\{(t,x)\mid  |x|\leq {\bar\delta}(T-t), s_m(x)\leq t\leq T\}.
\end{align*}
for small ${\bar\delta}>0$, $T>0$. Here ${\bar \delta}>0$ and $T>0$ play the role of $\ep>0$ in \eqref{eq:Ome:to:0}.

For $\omega_j$, $j=1,\ldots,m$ we take $\phi=\phi_j^{\pm}(t,x)=\pm(t-\sigma_j(x))$. For $\omega(T)$ we take $\phi=\phi_{m+1}=t-s_m(x)$ if $n\geq 1$. Turn to $\omega(T)$ with $n=0$. Without restrictions one can assume $\alpha>0$ and $\dif_t\alpha>0$ in $\omega(T)$ (see \cite[Lemma 2.2]{N1}) and we take $\phi=\alpha(t,x)$. 

\bigskip
\noindent
{\bf Acknowledgment.} This work was supported by 
JSPS KAKENHI Grant Number JP20K03679.

%%%%%%%%%%%%%%%%%%%%%%%%%%%

\end{document}